\documentclass[11pt,a4paper]{article}
\pdfoutput=1

\usepackage{lmodern}
\usepackage[normalem]{ulem}
\usepackage{titlesec}
\titleformat{\section}
  {\normalfont\fontsize{18pt}{10pt}\selectfont\scshape}
  {\thesection}
  {0.7em}
  {}
\titleformat{\subsection}
  {\normalfont\fontsize{14pt}{8pt}\selectfont\scshape}
  {\thesubsection}
  {0.6em}
  {}

\usepackage{abraces}
\usepackage[utf8]{inputenc}
\usepackage{amssymb}
\usepackage{amsmath}
\usepackage{amsfonts}
\usepackage{bbm}
\usepackage{amsthm}
\usepackage{mathrsfs}
\usepackage{hyperref}
\usepackage{color}
\usepackage[margin=2.5cm]{geometry}
\usepackage[all,cmtip]{xy}
\usepackage{graphicx}
\usepackage{float}
\usepackage{varwidth}
\usepackage{comment}
\usepackage{enumitem}
\usepackage{cancel}

\usepackage{bm}
\usepackage{mathtools}

\usepackage{svg}

\usepackage{upgreek}
\usepackage{rotating}

\usepackage{tikz}
\usetikzlibrary{shapes.geometric}

\usepackage{tikz}
\usetikzlibrary{cd,nfold,matrix,arrows,calc,decorations.pathmorphing,fit,shapes.geometric,decorations.pathreplacing,positioning}
\tikzset{Rightarrow/.style={double equal sign distance,>={Implies},->},
triple/.style={-,preaction={draw,Rightarrow}},
quadruple/.style={preaction={draw,Rightarrow,shorten >=0pt},shorten >=1pt,-,double,double
distance=0.2pt}}

\usepackage{quiver}

\mathchardef\mhyphen="2D

\definecolor{darkred}{rgb}{0.8,0.1,0.1}
\hypersetup{
	colorlinks=true,         
	linkcolor=darkred,
	citecolor=blue,
}

\newtheoremstyle{scstyle}%
  {0.4em}   
  {0.4em}   
  {}      
  {}      
  {\fontsize{12pt}{8pt}\selectfont\scshape} 
  {.}     
  {.5em}  
  {}      

\theoremstyle{scstyle}
\newtheorem{theo}{Theorem}[section]
\newtheorem{lem}[theo]{Lemma}
\newtheorem{propo}[theo]{Proposition}
\newtheorem{cor}[theo]{Corollary}

\theoremstyle{scstyle}
\newtheorem{defi}[theo]{Definition}

\newtheorem{exx}[theo]{Example}
\newenvironment{ex}
{\pushQED{\qed}\exx}
{\popQED\endexx}

\newtheorem{remm}[theo]{Remark}
\newenvironment{rem}
{\pushQED{\qed}\remm}
{\popQED\endremm}

\newtheorem{constrr}[theo]{Construction}
\newenvironment{constr}
{\pushQED{\qed}\constrr}
{\popQED\endconstrr}

\numberwithin{equation}{section}

\def\colim{\mathrm{colim}}
\def\nn{\nonumber}

\def\S{\mathrm{S}}
\def\rev{\mathrm{rev}}

\def\s{\rightsquigarrow}
\def\i{\tilde{i}}
\def\j{\tilde{j}}
\def\k{\tilde{k}}
\def\l{\tilde{l}}
\def\m{\tilde{m}}

\def\p{\tilde{p}}
\def\q{\tilde{q}}

\def\bbK{\mathbb{K}}

\def\bbN{\mathbb{N}}
\def\bbZ{\mathbb{Z}}

\def\g{\mathfrak{g}}

\def\KZ{\mathrm{KZ}}

\def\hom{\underline{\mathrm{hom}}}
\def\End{\mathrm{End}}
\def\Der{\mathrm{Der}}

\def\Imm{\mathrm{Im}}

\def\Sym{\mathrm{Sym}}

\def\id{\mathrm{id}}
\def\Id{\mathrm{Id}}

\def\D{\mathrm{d}}

\def\dR{\mathrm{dR}}

\def\dim{\mathrm{dim}}

\def\oone{\mathbbm{1}}
\def\op{\mathrm{op}}

\def\Vec{\mathsf{Vec}}

\def\Ho{\mathsf{Ho}}
\def\Ch{\mathsf{Ch}}
\def\dgMod{\mathsf{dgMod}}
\def\Rep{\mathsf{Rep}(\g)}

\def\C{\mathsf{C}}
\def\D{\mathsf{D}}

\def\K{\mathsf{K}}

\def\dgCat{\mathsf{dgCat}}

\def\ad{\mathrm{ad}}

\def\sf{\mathrm{sf}}

\def\CC{\mathcal{C}}
\def\DD{\mathcal{D}}

\def\Pol{\mathrm{Pol}}
\def\CE{\mathrm{CE}}

\newcommand\numberthis{\addtocounter{equation}{1}\tag{\theequation}}

\newcommand{\pb}[2]{[\![#1,#2]\!]}

\DeclareMathOperator*{\Mwedge}{\text{\raisebox{0.25ex}{\scalebox{0.7}{${\bigwedge}_{}$}}}}

\DeclareMathOperator*{\smallbox}{\text{\raisebox{0.15ex}{\scalebox{0.7}{$\boxtimes$}}}}
\DeclareMathOperator*{\cmp}{\text{\raisebox{0.2ex}{\scalebox{0.7}{$\otimes$}}}}

\makeatletter
\newcommand{\xRrightarrow}[2][]{\ext@arrow 0359\Rrightarrowfill@{#1}{#2}}
\newcommand{\Rrightarrowfill@}{\arrowfill@\equiv\equiv\Rrightarrow}
\newcommand{\xLleftarrow}[2][]{\ext@arrow 3095\Lleftarrowfill@{#1}{#2}}
\newcommand{\Lleftarrowfill@}{\arrowfill@\Lleftarrow\equiv\equiv}
\makeatother

\def\sk{\vspace{2mm}}

\makeatletter
\let\@fnsymbol\@alph
\makeatother

\title{%
Homotopy Lie algebras and coherent infinitesimal 2-braidings}

\author{%
Cameron Kemp\vspace{4mm}\\
{\small School of Mathematical Sciences, University of Nottingham,}\\
{\small University Park, Nottingham NG7 2RD, United Kingdom.}\vspace{4mm}\\
{\small \begin{tabular}{ll}
Email: & \href{mailto:cameron.kemp@nottingham.ac.uk}{\texttt{cameron.kemp@nottingham.ac.uk}}
\vspace{2mm}
\end{tabular}
}
}

\begin{document}

\maketitle

\begin{abstract}
\noindent Given a homotopy Lie algebra (i.e. an $L_\infty$-algebra) $\g$, we show concretely how the Lada-Markl $\g$-modules (i.e. representations) assemble into a symmetric monoidal dg-category. Considering the homotopy 2-category of that dg-category, we construct infinitesimal 2-braidings from 2-shifted Poisson structures then show that such infinitesimal 2-braidings are coherent in Cirio and Faria Martins' sense. We then explicitly determine the differential of the Chevalley-Eilenberg algebra associated with a finite-dimensional homotopy Lie algebra and construct the symmetric monoidal dg-equivalence between the category of representations and the category of semi-free dg-modules over the Chevalley-Eilenberg algebra.
\end{abstract}
\vspace{-3mm}

\vspace{-2mm}

\vspace{-2mm}


\tableofcontents
\newpage
\subsection*{Conventions and notations}
\begin{itemize}
\item All vector spaces, algebras, etc., will be over a fixed field $\bbK$ of characteristic $0$.
\item Every graded vector space will be graded by the integers $\bbZ$. Every differential structure on such a graded vector space will increase the degree by $1$, i.e. we use cochain conventions. 
\item For $i\in\bbZ$, we define $\i:=i-1$. For vector spaces $\g$, $U$ and $V$, we will denote identity maps as, e.g. $\id_{\g^{\otimes(\i-1)}\otimes\,U\otimes V}=:1_{\i-1,UV}$ and swaps as, e.g. $s_{\g^{\otimes(\i-1)},U}=:s_{\i-1,U}$ for space reasons.
\item For a mapping/equation, input elements will be assumed to be of homogeneous degree.
\item CDGA: commutative differential graded algebra, DGLA: differential graded Lie algebra.
\item Every monoidal category of structured sets is assumed to be strict\footnote{Should the reader be uneasy with this convention then rest be assured that the only associators and unitors that would crop up anyway would be the canonical ones from the category of vector spaces. This data is not only cumbersome to carry around but also inconsequential from the standpoint of our overarching story.} under Schauenburg's refined strictification theorem \cite[Theorem 4.3]{Schauenburg}.
\item We will work with many monoidal categories in this paper hence we list them here:
\begin{enumerate}
\item $(\Ch,\otimes,\bbK)$: cochain complexes, the tensor product and the ground field in degree 0 with trivial differential.
\item $(\dgCat,\cmp,\K)$: $\Ch$-categories, the local tensor product and the $\Ch$-category with a single object and hom given by $\bbK$.
\item $(\Ch^{[-1,0]},\boxtimes,\bbK)$: cochain complexes concentrated in degrees $\{-1,0\}$ and the truncated tensor product.
\item $(\dgCat^{[-1,0]},\smallbox,\K)$: $\Ch^{[-1,0]}$-categories and the local truncated tensor product.
\item $(\Rep,\odot,\bbK)$: representations of a homotopy Lie algebra $\g$, the odot product and the ground field in degree 0 with the trivial structure of representation.
\item $({}_A\dgMod,\otimes_A,A)$: dg-modules over a CDGA $A$, the relative tensor product and $A$ regarded as a dg-module over itself.
\end{enumerate}
\item Boldface is used to highlight the term being defined by surrounding text whereas italics are used for emphasis. For example, a \textbf{cochain 2-category} is a $\Ch^{[-1,0]}$-category and infinitesimal 2-braidings are defined to be \textit{pseudo}natural rather than 2-natural.
\end{itemize}

\subsection*{Acknowledgements}
First and foremost, gratitude goes to my senpais, Alexander Schenkel and Robert Laugwitz, for many fruitful insights and enlightening discussions over the 2-year course of this paper's formation. Not only would the content of Section \ref{sec:prelims} be largely absent without their attentive mindfulness but Section \ref{sec:weak reps} would be utterly hopeless and Section \ref{sec:geometric model} would be wrong in the most frustratingly subtle ways. 
\sk

I would like to also thank João Faria Martins for many conversations over the past year during which he raised several interesting questions, particularly the potential relationship between a deformation quantisation of the symmetric monoidal structure on the dg-category $\Rep$ and construction of a higher-dimensional analogue of topological invariants such as Kontsevich's Vassiliev invariant.

\newpage
\section{Introduction}
The history of $L_\infty$-algebras and their genesis is a highly intricate story, as detailed in Stasheff's retrospective analysis \cite[Section 4]{Sta19}. Technically, their concept made its first appearance in \cite{SS85} though their modern homological formulation in terms of a family of brackets on a graded vector space was only later introduced in \cite{Sta92,LS93}. The homological formulation makes it evident that $L_\infty$-algebras are a natural generalisation of DGLAs; this should not be too surprising given the reputation of DGLAs in deformation theory together with the fact that such a formulation was motivated by, and intended as an abstraction from, examples arising in BRST quantisation in closed string field theory due to Witten and Zwiebach \cite{WZ92,Zwi93}. To this end, the connection between $L_\infty$-algebras and the more general BV formalism \cite{BV,BF} was analysed later by Stasheff \cite{Sta97,Sta97'} and others \cite{Jurco}.
\sk

Shortly after the development of the homological formulation, it was recognised by Lada and Markl \cite[Theorem 2.3]{LM95} that finite-dimensional $L_\infty$-algebras are in one-to-one correspondence with augmented finitely generated semi-free CDGAs (Chevalley-Eilenberg algebras). A major payoff of this correspondence is that it allows one to naturally define homomorphisms between $L_\infty$-algebras via the established (and intuitively simple) definition of a homomorphism between CDGAs. Taking it further, one can then define a representation of an $L_\infty$-algebra $\g$ on a cochain complex $V$ as an $L_\infty$-algebra homomorphism from $\g$ to the DGLA $\End(V)$. In the classical theory of Lie algebras, a representation $\rho:\g\to\End(V)$ is equivalently described by an action $\rhd:\g\otimes V\to V$ subject to the condition that action by the commutator is equal to the commutator of the action. Ergo, in this $L_\infty$-algebra context, one anticipates that a representation is equivalently described by an infinite tower of higher and higher action maps subject to a homotopy coherence condition. In fact, Lada and Markl started from this latter description \cite[Definition 5.1]{LM95} then independently introduced the definition of a homomorphism from an $L_\infty$-algebra to a DGLA \cite[Definition 5.2]{LM95}; their main result \cite[Theorem 5.4]{LM95} was that these two definitions indeed coincide.
\sk

Recall the classical result that if $\g$ is a Lie algebra and $V$ is a representation then $\g\oplus V$ inherits a canonical Lie algebra structure; this result was generalised to $L_\infty$-algebras and their representations in Lada's short note \cite{Lada}. As stated in the introduction, that short note was ``\textit{motivated by a problem in mathematical physics originally encountered in \cite{BBvD} and addressed in \cite{FLS}. In classical gauge field
theory, one encounters representations of Lie algebras in the guise of the Lie algebra of gauge parameters acting on the Lie module of fields for the theory. However,
as in \cite{BBvD}, these Lie structures occasionally appear up to homotopy}." Now recall another classical result: given a Lie algebra $\g$, two representations $V,W$ and a homomorphism $f:V\to W$, one can define a Lie algebra homomorphism $\id_\g\oplus f:\g\oplus V\to\g\oplus W$. With a view to generalising this simple result to the $L_\infty$-algebra context, Allocca \cite{thesis,Allocca} defined homomorphisms of $L_\infty$-algebra representations then showed that the analogous result indeed holds.
\sk

To the best of our knowledge, composition of Allocca's homomorphisms has not been defined in the literature nor has their monoidal product. In fact, what is really desired is:
\begin{enumerate}
\item[(i)] A cochain complex of morphisms between any two representations such that Allocca's homomorphisms are 0-cocycles with respect to the hom differential.
\item[(ii)] An associative composition which is a cochain map between hom-complexes and admits units that are 0-cocycles.
\item[(iii)] A compatible symmetric monoidal structure on this dg-category $\Rep$. 
\end{enumerate}
This desideratum is deeply rooted in our long-term ambition of constructing `higher quantum groups', as explored in \cite{KLS25,KLS24,Kemp25,Kemp25'}. Let us now recap that work so as to clearly reveal the necessity of the current paper.
\sk

It was shown in \cite{KLS25} how a 2-shifted Poisson structure $\pi$ on a semi-free CDGA $A$ induces an infinitesimal 2-braiding $t^\pi$ on the homotopy 2-category of the symmetric monoidal dg-category of semi-free $A$-dg-modules, $\Ho_2\big({}_A\dgMod^\sf\big)$. This infinitesimal 2-braiding $t^\pi$ was integrated \cite[Subsection 3.3]{KLS25} to second order in the deformation parameter $\hbar$ by choosing the ansatz braiding as $\sigma:=\gamma\,e^{\frac{\hbar}{2}t^\pi}$ and the ansatz associator as $\alpha:=\Phi_\KZ(t^\pi_{12},t^\pi_{23})$, where $\gamma$ is the symmetric braiding on ${}_A\dgMod^\sf$ and $\Phi_\KZ$ is Drinfeld's Knizhnik-Zamolodchikov series \cite[Corollary XIX.6.5]{Kassel}. It was shown that, at that order, the pentagon axiom was satisfied on the nose yet the hexagon axioms were obstructed by modifications. This investigation was followed up in \cite[Section 5]{Kemp25} where it was shown that the modifications satisfy the requisite axioms of a braided monoidal 2-category if the infinitesimal 2-braiding is coherent and totally symmetric, in Cirio and Faria Martins' sense \cite[Definitions 17 and 18, respectively]{Joao}.
\sk

Cirio and Faria Martins introduced such definitions surrounding an infinitesimal 2-braiding so that they could build from it a KZ 2-connection with nice geometric properties, as developed in their earlier work \cite{Joao1,Joao2}. It was shown recently \cite{Kemp25'} that one can lift the above-mentioned obstructions to all orders (i.e. construct a pair of hexagonator series) by studying the higher-categorical parallel transport of Cirio and Faria Martins' KZ 2-connection. By assuming again that the infinitesimal 2-braiding is coherent and totally symmetric, it was demonstrated that the braiding, associator and pair of hexagonator series together satisfy the only axiom of a braided monoidal 2-category which does not include the pentagonator, i.e. the Breen polytope axiom \cite[(2.40)]{Kemp25}. The condition of total symmetry was easy to show for our concrete infinitesimal 2-braiding $t^\pi$ whereas we could only conjecture \cite[Remark 5.28]{Kemp25} that $t^\pi$ being coherent relates to the third-weight component of the Maurer-Cartan equation that the shifted Poisson structure $\pi$ must satisfy \cite[Remark 2.7]{KLS25}. Let us expand on this latter point and bring light to subtleties that have been glossed over in the story above.   
\sk

Given a CDGA $A=(A,\delta)$, the derivations $\Der(A)\subseteq\hom(A,A)$ are canonically equipped with the structure of a DGLA where the Lie bracket is the graded commutator $[\cdot,\cdot]$ and the differential is the internal hom differential $\partial:=[\delta,\cdot\,]$. Given $n\in\bbZ$, the CDGA of completed $n$-shifted polyvectors is defined by $\widehat{\Pol}(A,n):=\prod_{w\in\bbN}\Sym^w_A\big(\Der(A)[-n-1]\big)$, where $w\in\bbN$ is called the \textbf{weight}. There is a canonical Poisson $(n+2)$-algebra
structure on $\widehat{\Pol}(A,n)$, i.e. the Schouten-Nijenhuis bracket $\{\cdot,\cdot\}$, with respect to which one defines the \textbf{Maurer-Cartan equation} for a completed $n$-shifted polyvector $P\in\widehat{\Pol}(A,n)$ as
\begin{align}\label{eq:MC eqn in Pol}
\partial(P) + \tfrac{1}{2}\,\{P,P\}\,=\,0\quad.
\end{align}
Following \cite{CPTVV,Pridham,PridhamOutline}, we define an \textbf{$n$-shifted Poisson structure} on the CDGA $A$ as a completed $n$-shifted polyvector $\pi$ of cohomological degree $n+2$ and weight $w\geq2$ which satisfies the Maurer-Cartan equation \eqref{eq:MC eqn in Pol}. Expanding the $n$-shifted Poisson structure into its weight components as $\pi:=\sum_{w=2}^\infty\pi_w\in\widehat{\Pol}(A,n)^{n+2}$ where $\pi_w\in\Sym^w_A\big(\Der(A)[-n-1]\big)^{n+2}$, we decompose the Maurer-Cartan equation \eqref{eq:MC eqn in Pol} into such weight components as:
\begin{alignat}{6}
\partial(\pi_2)\,=&\,\,0\quad&&,\label{eq:2nd weight MC in Pol}\\
\partial(\pi_3)+\tfrac{1}{2}\,\{\pi_2,\pi_2\}\,=&~0&&,\label{eq:3rd weight MC in Pol}\\
\vdots\nn
\end{alignat}
Given a semi-free CDGA $A$ and a 2-shifted Poisson structure $\pi\in\widehat{\Pol}(A,2)^4$, the infinitesimal 2-braiding $t:\otimes_A\Rightarrow\otimes_A$ of \cite[Theorem 3.10]{KLS25} was constructed from \textit{only} the biderivation component $\pi_2\in\Sym^2_A\big(\Der(A)[-3]\big)^4\cong\Sym^2_A\big(\Der(A)[-1]\big)^0$ and the proof of that theorem made use of \textit{only} the weight 2 component of the Maurer-Cartan equation \eqref{eq:2nd weight MC in Pol}. For this reason, it would have been clearer to have called that paper ``Infinitesimal 2-braidings from \textit{the biderivation component} of a 2-shifted Poisson structure" and to have denoted the induced infinitesimal 2-braiding as $t^{\pi_2}$. The intuition behind linking the weight 3 component of the Maurer-Cartan equation \eqref{eq:3rd weight MC in Pol} to coherency of $t^{\pi_2}$ can be understood by first recalling, from the classical context, Cartan's formula \cite[Proposition
3.6]{Poisson}. Recall that this formula relates the internal product action of multiderivations on differential forms $\iota$ to: the classical Schouten-Nijenhuis bracket $[\cdot,\cdot]_\mathrm{SN}$, the de Rham differential $d_{\dR}$ and the graded commutator bracket $[\cdot,\cdot]$, specifically as
\begin{equation}\label{eq:Cartan's formula}
\big[\iota_P,[d_{\dR},\iota_Q]\big]=\iota_{[P,Q]_\mathrm{SN}}\quad.
\end{equation}
If one could generalise Cartan's formula \eqref{eq:Cartan's formula} to the CDGA context, then the coherency condition of $t^{\pi_2}\sim\iota_{\pi_2}(\,\cdots)$ would boil down to the requirement that 
\begin{equation}\label{eq:rough condition of t being coherent}
t^{\pi_2}_{12}+t^{\pi_2}_{23}+t^{\pi_2}_{13}\sim\big[\iota_{\pi_2},[d_{\dR},\iota_{\pi_2}]\big]\big(\cdots\big)=0\quad.
\end{equation}
The triderivation component $\pi_3\in\Sym^3_A\big(\Der(A)[-3]\big)^4\cong\Sym^3_A\big(\Der(A)[-1]\big)^{-2}$ induces a degree $-2$ morphism $t^{\pi_3}\sim\iota_{\pi_3}(\,\cdots)$ in ${}_A\dgMod^\sf$ thus
\begin{equation}
d_{\hom}(t^{\pi_3})\sim\iota_{\partial(\pi_3)}(\,\cdots)
\end{equation}
vanishes in the homotopy 2-category $\Ho_2\big({}_A\dgMod^\sf\big)$ and the coherence condition \eqref{eq:rough condition of t being coherent} is indeed satisfied because of the weight 3 component of the Maurer-Cartan equation \eqref{eq:3rd weight MC in Pol} together with the assumed given CDGA generalisation of Cartan's formula \eqref{eq:Cartan's formula}.
\sk

Rather than attempting to prove that a CDGA generalisation of Cartan's formula \eqref{eq:Cartan's formula} holds, the present paper offers an alternative route to constructing a coherent (and totally symmetric) infinitesimal 2-braiding from a 2-shifted Poisson structure by explicitly defining the symmetric monoidal dg-category $\Rep$ via the graphical calculus formalism developed in \cite{KLS24}. That formalism was developed to characterise the shifted Poisson structures on a generic semi-free CDGA in terms of a family of homotopically-coherent maps adjoined to the underlying (curved) $L_\infty$-algebra. The inspiration was Safronov's result \cite{Safronov} that, in the case of an ordinary Lie algebra $\g$, the 1- and 2-shifted Poisson structures are given by quasi-Lie bialgebra structures and invariant symmetric tensors $\pi_2\in(\Sym^2\g)^{\ad\,\g}$, respectively, i.e. the semiclassical data associated with Drinfeld-Jimbo quantum groups.
\sk

In brief, the graphical calculus formalism of \cite{KLS24} works as follows: given a finitely generated semi-free CDGA $A=\big(\Sym(\g^*[-1]),\delta\big)$, the $n$-shifted polyvectors of weight $w\in\bbN$ decompose into arity $p\in\bbN$ components which can be rewritten as $P:\bigwedge^p\g\to(\Sym_\pm^w\g)\big[|P|_{\Pol}-p-wn\big]$, where $\Sym_\pm$ is symmetric/alternating if $n$ is even/odd. The Schouten-Nijenhuis bracket $\{\cdot,\cdot\}$ translates into a commutator of a bullet product composition operation \eqref{eq:def of Schouten by bullet}. In particular, the differential $\delta$ is itself an $n$-shifted polyvector of weight 1 and cohomological degree $n+2$ which satisfies the Maurer-Cartan equation \eqref{eq:MC eqn in Pol}, i.e. $\delta\circ\delta=0$. This then corresponds to a family of maps $\{\ell^p:\bigwedge^p\g\to\g[2-p]\}_{p\in\bbN}$ subject to the curved Jacobi identity, i.e. we recover the well-known one-to-one correspondence. Not only does this bullet product allow one to abstract the definition of an $n$-shifted Poisson structure in an $L_\infty$-algebra $\g$ of arbitrary dimension, but it also allows one to naturally define the structure of a representation of $\g$ together with the dg-category structure of $\Rep$, see Remark \ref{rem:comparison to Allocca}.

\newpage
\section{Preliminaries}\label{sec:prelims}
The aim of these preliminaries is to make this paper self-contained by fixing notation and basic terms. Subsection \ref{sub:symmetric monoidal dg-cats} begins by recalling some elementary concepts of homological algebra needed to describe the symmetric monoidal category of cochain complexes, see \cite[Chapter 1]{Weibel}. We then apply these concepts to the context of enriched category theory (i.e. we discuss dg-category theory) with which we assume some familiarity, see e.g. \cite{Kelly} for a comprehensive treatment, 
or \cite[Section 1.3]{Yau} and \cite[Chapter 3]{Riehl} for more concise introductions. 
\sk

Subsection \ref{sub:homotopy Lie algebras and SPS} recalls the bullet product from \cite{KLS24} and uses it to fix the notion of homotopy Lie algebra in Definition \ref{defi: L infinity algebra}. We then take the characterisation result \cite[Corollary 3.10]{KLS24} as a \textit{starting point} by introducing the notion of a shifted Poisson structure in a homotopy Lie algebra in Definition \ref{def:sps in homotopy lie algebra}. Subsection \ref{sub:inf2bra} recalls from \cite[Appendix A.2]{KLS25} the idea of cochain complexes concentrated in degrees $\{-1,0\}$ and their truncated tensor product; the resulting category $\Ch^{[-1,0]}$ is used as a basis for enrichment. We then recall from \cite[Subsection 2.1]{Kemp25} pseudonatural transformations (between $\Ch^{[-1,0]}$-functors between $\Ch^{[-1,0]}$-categories) together with their various compositions. We then recall from \cite{Kemp25} the definition of an infinitesimal 2-braiding together with the simplifications of Cirio and Faria Martin's \cite{Joao} notions of coherency and total symmetry.

\subsection{Symmetric monoidal dg-categories}\label{sub:symmetric monoidal dg-cats}
We say a vector space $V$ is a \textbf{graded vector space}  if it admits a decomposition into a direct sum $V^\sharp=\bigoplus_{m\in\bbZ}V^m$ of disjoint subspaces labelled by integers $m\in\bbZ$ (called \textbf{degree}). Given another graded vector space $W$, a linear map $f:V\to W$ is said to be a \textbf{graded linear map} if it preserves degrees, i.e. $\Imm(f_{|V^m})\subseteq W^m$ for all $m\in\bbZ$. 
\begin{defi}\label{def:shifted graded vector space}
Given $n\in\bbZ$ and a graded vector space $V$, the \textbf{$n$-shifted graded vector space} $V[n]$ is defined by $V[n]^m:=V^{m+n}$ for all $m\in\bbZ$.
\end{defi}
A \textbf{cochain complex} $V=(V,d)$ is a graded vector space $V$ with a graded linear map $d:V\to V[1]$
(called \textbf{differential}) which squares to zero, $d\circ d=0$. For $m\in\bbZ$, the subspace of \textbf{$m$-cocycles} is defined as $Z^m(V):=\ker(d_{|V^m})\subseteq V^m$ and the subspace of \textbf{$m$-coboundaries} is defined as $B^m(V):=\Imm(d_{|V^{m-1}})\subseteq V^m$. Given another cochain complex $W$, a \textbf{cochain map} $f:V\to W$ is a graded linear map which commutes with the differentials, $f\circ d=d\circ f$.
\begin{rem}\label{rem:cochain map restricts to coboundaries/cocycles}
A cochain map sends coboundaries to coboundaries and cocycles to cocycles.
\end{rem}
The category
of cochain complexes and cochain maps is denoted by $\Ch$, it carries the following closed symmetric monoidal structure:
\begin{enumerate}
\item[(i)] Given cochain complexes $V$ and $W$, their \textbf{tensor product} $V\otimes W$ is defined by
\begin{subequations}\label{eqn:Chtensor}
\begin{align}
(V\otimes W)^m:=\bigoplus_{n\in\bbZ}\big(V^n\otimes W^{m-n}\big)
\end{align}
with the differential provided as
\begin{align}
d(v\otimes w):=(d v)\otimes w + (-1)^{|v|}v\otimes(d w)
\end{align}
\end{subequations}
where $|\cdot|$ indicates the degree.
\item[(ii)] The monoidal unit is given by the ground field $\bbK$ regarded as a cochain complex concentrated in degree
$0$ with trivial differential.
\item[(iii)] The symmetric braiding $s:\otimes\Rightarrow\otimes^\op$ is defined by the \textbf{Koszul sign rule},
\begin{align}\label{eq:Koszul sign}
s_{V,W}:V\otimes W\to W\otimes V\qquad,\qquad v\otimes w\mapsto(-1)^{|v||w|}w\otimes v\quad.
\end{align}
\item[(iv)] The internal hom $\Ch[V,W]\in\Ch$ is given by
\begin{subequations}\label{eqn:Chhom}
\begin{align}
\Ch[V,W]^m:=\prod_{n\in\bbZ}\Vec[V^n,W^{n+m}]
\end{align}
with the differential structure as the \textbf{internal hom differential},
\begin{align}
\partial h:=d\circ h - (-1)^{|h|}h\circ d\quad.
\end{align}
\end{subequations}
\end{enumerate}
Note that cochain maps $f:V\to W$
are precisely the $0$-cocycles in $\Ch[V,W]$, i.e. elements $f\in\Ch[V,W]^0$ of degree $0$ 
satisfying $\partial f =0$. Given two cochain maps
$f,g:V\to W$, a \textbf{homotopy} $h:f\Rightarrow g$  is a degree $-1$ element $h\in\Ch[V,W]^{-1}$ such that $\partial h=f-g$, in which case the difference $f-g$ is a 0-coboundary. 
\begin{rem}\label{rem:tautological rewrite of graded linear map}
Shifting of degrees as in Definition \ref{def:shifted graded vector space} allows us to tautologically rewrite a homogeneous element $f\in\Ch[V,W]$ as a graded linear map $f:V\to W\big[|f|\big]$. 
\end{rem}
\begin{defi}
A category $\C$ is a \textbf{dg-category} if it has hom-complexes $\C[V,V']\in\Ch$ such that units are 0-cocycles $1_V\in Z^0\big(\C[V,V]\big)$ and composition is a cochain map. Given another dg-category $\C'$, we say a functor $F:\C\to\C'$ is a \textbf{dg-functor} it its morphism map is a cochain map. Given another dg-functor $F':\C\to\C'$, we say a natural transformation $\xi:F\Rightarrow F':\C\to\C'$ is a \textbf{dg-natural transformation} if $\xi_V\in Z^0\big(\C'[F(V),F'(V)]\big)$ for all $V\in\C$; if every component $\xi_V$ is an isomorphism then we say that $\xi$ is a \textbf{dg-natural isomorphism}.
\end{defi}
\begin{ex}\label{ex:Ch is a dg-category}
There is a dg-category of cochain complexes with hom-complexes given by the internal hom \eqref{eqn:Chhom}, we also denote this dg-category by $\Ch$ given that the context will always make it clear whether are referring to the category or the more general dg-category. 
\end{ex}
We denote by $\dgCat$ the 2-category of dg-categories, dg-functors and dg-natural transformations. As always in enriched category theory \cite[Section 1.4]{Kelly}, the 2-category $\dgCat$ has a canonical
symmetric monoidal structure:
\begin{enumerate}
\item[(i)] The dg-category $\C\cmp\D$ has objects given by 2-tuples, $(V,W)$ where $V\in\C$ and $W\in\D$, with hom-complexes given by the tensor product 
\begin{equation}\label{eq:tensor prod between hom complexes}
\big(\C\cmp\D\big)[(V,W),(V',W')]:=\C[V,V']\otimes\D[W,W']\quad.
\end{equation}
The (cochain) composition of $\C\cmp\D$ is simply given by
\begin{equation}\label{eq:interchange for tensor product}
(f'\otimes g')(f\otimes g):=(-1)^{|g'||f|}f'f\otimes g'g
\end{equation}
hence the units are given by $1_{V,W}:=1_V\otimes1_W$.
\item[(ii)] The dg-functor $F\cmp G:\C\cmp\D\to\C'\cmp\D'$ has object map $\big(F\cmp G\big)(V,W):=\big(F(V),G(W)\big)$ and morphism map
\begin{equation}
f\otimes g\mapsto F(f)\otimes G(g)\quad.
\end{equation}
\item[(iii)] The dg-natural transformation $\xi\cmp\zeta:F\cmp G\Rightarrow F'\cmp G':\C\cmp\D\to\C'\cmp\D'$ is simply given by $(\xi\cmp\zeta)_{V,W}:=\xi_V\otimes\zeta_W$.
\item[(iv)] The monoidal unit $\K$ is the single object dg-category with hom-complex given by the ground field $\bbK$ and cochain composition simply given by multiplication.
\item[(v)] The symmetric braiding $S$ has components $S_{\C,\D}:\C\cmp\D\cong\D\cmp\C$ given simply by: 
\begin{subequations}
\begin{alignat}{6}
(V,W)&\mapsto(W,V)\quad&&,\\
f\otimes g&\mapsto (-1)^{|f||g|}g\otimes f&&.
\end{alignat}
\end{subequations}
\end{enumerate}
This symmetric strict monoidal 2-category $(\dgCat,\cmp,\K,S)$ allows us to provide the following fundamental definition.

\begin{defi}\label{def:symmetric strict monoidal dg-category}
A \textbf{symmetric strict monoidal dg-category} is a 4-tuple $(\C,\odot,I,\gamma)$ consisting of: a dg-category $\C$, a dg-functor $\odot:\C\cmp\C\to\C$, an object $I\in\C$ and a dg-natural isomorphism $\gamma:\odot\Rightarrow\odot^\op$ which is involutive, i.e.
\begin{equation}
\gamma_{V,W}=\gamma_{W,V}^{-1}\quad.
\end{equation}
This data is required to satisfy the axioms of strict associativity, strict unitality and \textit{the} hexagon axiom, 
\begin{align}\label{eq:hexagon axiom of a symmetric braiding}
\gamma_{U\odot V,W}=(\gamma_{U,W}\odot1_V)(1_U\odot\gamma_{V,W})\quad.
\end{align} 
\end{defi}
\begin{ex}
Following on from Example \ref{ex:Ch is a dg-category}, $(\Ch,\otimes,\bbK,s)$ is a symmetric strict monoidal dg-category; we express the dg-naturality of $s$ for later purposes,
\begin{align}\label{eq:naturality of symmetric braiding}
s_{V',W'}(f\otimes g)=(-1)^{|f||g|}(g\otimes f)s_{V,W}\quad.
\end{align}
\end{ex}
\begin{defi}\label{def:sym mon dg-equivalence}
Given symmetric strict monoidal dg-categories $(\C,\odot,I,\gamma)$ and $(\C',\odot',I',\gamma')$, a \textbf{symmetric strict monoidal dg-equivalence} is an essentially surjective fully faithful dg-functor $F:\C\to\C'$ such that:
\begin{subequations}
\begin{alignat}{6}
F(V\odot W)=\,&F(V)\odot' F(W)\quad&&,\\
F(f\odot g)=\,&F(f)\odot' F(g)&&,\\
F(I)=\,&I'&&,\\
F(\gamma_{V,W})=\,&\gamma'_{F(V),F(W)}&&.
\end{alignat}
\end{subequations}
\end{defi}
We close this subsection by introducing the following definition which weakens the notion of dg-natural transformation; we will need it for later purposes in Subsection \ref{subsec:Poisson inf2bra}.
\begin{defi}\label{def:dg-pseudonatural transformation}
Given dg-functors $F,G:\C\to\D$, a \textbf{dg-pseudonatural transformation} $\varrho:F\Rightarrow G:\C\to\D$ consists of the following two pieces of data:
\begin{enumerate}
\item[(i)] For each object $U\in\C$, a 0-cocycle $\varrho_U\in Z^0\big(\D[F(U),G(U)]\big)$.
\item[(ii)] For each pair of objects $U,V\in\C$, a homotopy $\varrho_{(\cdot)}:\C[U,V]\to\D[F(U),G(V)][-1]$.
\end{enumerate}
These two pieces of data are required to satisfy the following two axioms:
\begin{enumerate}
\item[(i)] For each morphism $f\in\C[U,V]$,
\begin{equation}\label{eq:pseudonaturality axiom of dg-pseudo}
G(f)\varrho_U-\varrho_VF(f)=\partial(\varrho_f)+\varrho_{\partial f}\quad.
\end{equation}
\item[(i)] For each pair of morphisms $f\in\C[U,V]$ and $g\in\C[V,W]$,
\begin{equation}\label{eq:derivation axiom of dg-pseudo}
\varrho_{gf}=\varrho_gF(f)+(-1)^{|g|}G(g)\varrho_f\quad.
\end{equation}
\end{enumerate}
\end{defi}

\subsection{Homotopy Lie algebras and shifted Poisson structures}\label{sub:homotopy Lie algebras and SPS}
\begin{defi}
For $j\in\bbN$ and $i_1,\ldots,i_j\in\bbN$, an \textbf{$(i_1,\ldots,i_j)$-shuffle} of $(1,\ldots,i_1+\ldots+i_j)$ is a permutation $\sigma\in\S_{i_1+\ldots+i_j}$ such that $\sigma(i_1+\ldots+i_k+1)<\sigma(i_1+\ldots+i_k+2)<\cdots<\sigma(i_1+\ldots+i_{k+1})$ for $0\leq k\leq j-1$. We denote the subset of such $(i_1,\ldots,i_j)$-shuffles as $\S_{i_1,\ldots,i_j}\subseteq\S_{i_1+\ldots+i_j}$. The \textbf{signature} of a permutation $\sigma$ is defined as
\begin{subequations}
\begin{align}
  \mathrm{sgn}(\sigma):=\begin{cases}
    \,\,\,1\,\,\,,\qquad&\mathrm{Even~number~of~transpositions}\\
    -1\,,\qquad&\mathrm{Odd~number~of~transpositions}
    \end{cases}\quad.
\end{align}
Given a graded vector space $\g$, we define the \textbf{$(i_1,\ldots,i_j)$-shuffler} as:
\begin{align}
\Sigma_{i_1,\ldots,i_j}:=\sum_{\sigma\in\S_{i_1,\ldots,i_j}}\mathrm{sgn}(\sigma)\sigma:\g^{\otimes i_1}\otimes\ldots\otimes\g^{\otimes i_j}&\to\g^{\otimes(i_1+\ldots+i_j)}\,,\label{subeq:(i,j)-shuffler}\\
(x_1\otimes\cdots\otimes x_{i_1})\otimes\ldots\otimes(x_{i_1+\ldots+i_{\j}+1}\otimes\cdots\otimes x_{i_1+\ldots+i_j})&\mapsto\sum_{\sigma\in\S_{i_1,\ldots,i_j}}\varepsilon_\sigma\,x_{\sigma(1)}\otimes\cdots\otimes x_{\sigma(i_1+\ldots+i_j)}\,,\nn
\end{align}
where $\varepsilon_\sigma:=\chi_\sigma\,\mathrm{sgn}(\sigma)$ and the sign $\chi_\sigma$ is the Koszul sign \eqref{eq:Koszul sign} coming from the permutation $\sigma$ of graded vector space elements $x\in\g$. We also define the $(i_1,\ldots,i_j)$-\textbf{symmetriser} as:
\begin{align}
\Sigma^+_{i_1,\ldots,i_j}:=\sum_{\sigma\in\S_{i_1,\ldots,i_j}}\sigma:\g^{\otimes i_1}\otimes\ldots\otimes\g^{\otimes i_j}&\to\g^{\otimes(i_1+\ldots+i_j)}\,,\label{subeq:(i,j)-symmetriser}\\
(x_1\otimes\cdots\otimes x_{i_1})\otimes\ldots\otimes(x_{i_1+\ldots+i_{\j}+1}\otimes\cdots\otimes x_{i_1+\ldots+i_j})&\mapsto\sum_{\sigma\in\S_{i_1,\ldots,i_j}}\chi_\sigma\,x_{\sigma(1)}\otimes\cdots\otimes x_{\sigma(i_1+\ldots+i_j)}\,,\nn
\end{align}
\end{subequations}
We define $\Sigma^-_{i_1,\ldots,i_j}:=\Sigma_{i_1,\ldots,i_j}$ and set $\Sigma_{i_1,\ldots,i_j}^\pm$ to 0 if any of the $i_k$ are $-1$. 
\end{defi}
\begin{rem}\label{rem:shuffler restriction is well-defined}
Restricting the $(i_1,\ldots,i_j)$-shuffler \eqref{subeq:(i,j)-shuffler} to the subspace $(\bigwedge^{i_1}\g)\otimes\ldots\otimes(\bigwedge^{i_j}\g)$ produces a well-defined map $\Sigma_{i_1,\ldots,i_j}:(\bigwedge^{i_1}\g)\otimes\ldots\otimes(\bigwedge^{i_j}\g)\to\bigwedge^{i_1+\ldots+i_j}\g$. Likewise, restricting the $(i_1,\ldots,i_j)$-symmetriser \eqref{subeq:(i,j)-symmetriser} to the subspace $(\Sym^{i_1}\g)\otimes\ldots\otimes(\Sym^{i_j}\g)$ produces a well-defined map $\Sigma^+_{i_1,\ldots,i_j}:(\Sym^{i_1}\g)\otimes\ldots\otimes(\Sym^{i_j}\g)\to\Sym^{i_1+\ldots+i_j}\g$. Within this context, we state the following relations which can be easily proved by evaluating on inputs. Given $i,j\in\bbN$, the shuffler \eqref{subeq:(i,j)-shuffler} satisfies
\begin{align}\label{eq:unitality of shuffler}
\Sigma_{i,\j}\otimes1_\g+(-1)^j(1_{\i}\otimes s_{j,\g})(\Sigma_{\i,j}\otimes1_\g)=\Sigma_{i,j}=1_\g\otimes \Sigma_{\i,j}+(-1)^i(s_{\g,i}\otimes1_{\j})(1_\g\otimes \Sigma_{i,\j}),
\end{align}
and, for $i_1,\ldots,i_j\in\bbN$ and $1\leq k\leq j-1$:
\begin{alignat}{6}
\Sigma_{i_1,\ldots,i_j}&=(1_{i_1+\ldots+i_{k-1}}\otimes\Sigma_{i_k,i_{k+1}}\otimes1_{i_{k+2}+\ldots+i_j})\Sigma_{i_1,\ldots,i_{k-1},i_k+i_{k+1},i_{k+2}\ldots,i_j}\,&&,\label{eq:associativity of shuffler}\\
\Sigma_{i_1,\ldots,i_j}&=(-1)^{i_ki_{k+1}}(1_{i_1+\ldots+i_{k-1}}\otimes s_{i_{k+1},i_k}\otimes1_{i_{k+2}+\ldots+i_j})\Sigma_{i_1,\ldots,i_{k-1},i_{k+1},i_k,i_{k+2},\ldots,i_j}\,&&.\label{eq:symmetry of shuffler}
\end{alignat}
\end{rem}
Setting 
\begin{align}\label{eq:Sym_pm}
\Sym_\pm^i\g:=\begin{cases}
\Sym^i\g&,\quad\text{for $n$ even}\\
{\bigwedge}^i\g&,\quad\text{for $n$ odd}\\
\end{cases}
\end{align}
and 
\begin{align}
\Sigma^\pm_{i,j}:=\begin{cases}
\Sigma^+_{i,j}&,\quad\text{for $n$ even}\\
\Sigma_{i,j}&,\quad\text{for $n$ odd}\\
\end{cases}\quad,
\end{align}
we recall from \cite[(3.28)]{KLS24} the bullet product $\bullet$ of graded linear maps between skew-/symmetric powers of $\g$. Given $p,q\in\bbN$, $P:\bigwedge^j\g\to(\Sym_\pm^p\g)[|P|]$ and $Q:\bigwedge^k\g\to(\Sym_\pm^q\g)[|Q|]$, we define their \textbf{bullet product} $P\bullet Q:\big(\bigwedge^{\j}\g\big)\otimes\big(\bigwedge^k\g\big)\to(\Sym_\pm^p\g)\otimes(\Sym_\pm^{\q}\g)\big[|P|+|Q|\big]$ as
\begin{align}\label{eq:definition of bullet product}
P\bullet Q:=(-1)^{(|P|+\j)\q n+|P|\k}\Sigma_{p,\q}^\pm(P\otimes1_{\q})(1_{\j}\otimes Q)\Sigma_{\j,k}
\end{align}
and, as in Remark \ref{rem:shuffler restriction is well-defined}, this produces a well-defined map 
\begin{equation}
P\bullet Q:{\Mwedge}^{\j+k}\g\to(\Sym_\pm^{p+\q}\g)\big[|P|+|Q|\big]\quad.
\end{equation}
The \textbf{Schouten-Nijenhuis} bracket is constructed as the commutator bracket of the bullet product \cite[(3.17)]{KLS24}, 
\begin{align}\label{eq:def of Schouten by bullet}
\pb{P}{Q}:=P\bullet Q-(-1)^{(|P|+\p n+\j)(|Q|+\q n+\k)}Q\bullet P\quad.
\end{align}
\begin{rem}\label{rem:bullet product}
Restricting to the case where $P:\bigwedge^j\g\to\g[|P|]$ and $Q:\bigwedge^k\g\to\g[|Q|]$, we have
\begin{align}\label{eq:bullet composition operation}
P\bullet Q\overset{\eqref{eq:definition of bullet product}}{=}(-1)^{|P|\k}P(1_{\j}\otimes Q)\Sigma_{\j,k}\overset{\eqref{eq:symmetry of shuffler},\eqref{eq:naturality of symmetric braiding}}{=}(-1)^{(|P|+\j)\k}P(Q\otimes1_{\j})\Sigma_{k,\j}\quad.
\end{align}
\end{rem}
As in \cite[Corollary 4.1]{KLS24}, Remark \ref{rem:bullet product} allows us to naturally produce the standard definition of an $L_{\infty}$-algebra from, say, \cite[(7a)]{Jurco} or \cite[Remark 3.6]{Schnitzer}.
\begin{defi}\label{defi: L infinity algebra}
Given a family of graded linear maps
\begin{subequations}
\begin{align}\label{eq:data of l alg}
\pi_1:=\left\{\pi^i_1:=(-1)^{\i}\ell^i:{\Mwedge}^i\g\to\g[2-i]\right\}_{i\geq1}\quad,
\end{align}
$\g$ is a \textbf{homotopy Lie algebra} if $\pb{\pi_1}{\pi_1}=0$, i.e. the \textbf{generalised Jacobi identity} is satisfied: 
\begin{align}\label{eq:generalised Jacobi identity}
\forall i\geq1\,,\qquad\qquad\sum_{\j+k=i}(-1)^{\j}\ell^j(\ell^k\otimes1_{\j})\Sigma_{k,\j}=0
\end{align}
which reads in terms of inputs $x_1,\ldots,x_i\in\g$ as
\begin{align}
\sum_{\j+k=i}\sum_{\sigma\in\S_{k,\j}}\varepsilon_\sigma(-1)^{\j}\ell^j\big(\ell^k(x_{\sigma(1)},\ldots,x_{\sigma(k)}),x_{\sigma(k+1)},\ldots,x_{\sigma(i)}\big)=0\quad.
\end{align}
\end{subequations}
\end{defi}
\begin{rem}\label{rem:low arity data of L_inf-algebra}
Let us inspect the low arity data and the requirements that the generalised Jacobi identity \eqref{eq:generalised Jacobi identity} imposes:
\begin{enumerate}
\item $\ell^1:\g\to\g[1]$ is a differential on $\g$ because $\ell^1\circ\ell^1=0$.
\item The \textbf{bracket} $\ell^2:\bigwedge^2\g\to\g$ is a cochain map because $\ell^1\circ\ell^2=\ell^2\circ\ell^1$.
\item The bracket $\ell^2$ satisfies the Jacobi identity up to the internal hom differential of the \textbf{Jacobiator} $\ell^3:\bigwedge^3\g\to\g[-1]$, i.e. $\ell^2\circ(\ell^2\otimes\id_\g)\circ \Sigma_{2,1}=\ell^1\circ\ell^3+\ell^3\circ\ell^1$. 
\end{enumerate}
\end{rem}
\begin{ex}\label{ex:DGLA and ordinary Lie algebra}
A \textbf{DGLA} is a homotopy Lie algebra with $\ell^i=0$ for $i\geq3$ whereas a \textbf{Lie algebra} is a homotopy Lie algebra with $\g=\g^0$. 
\end{ex}
Taking the characterisation result \cite[Corollary 3.10]{KLS24} as our starting point, we introduce the following definition.
\begin{defi}\label{def:sps in homotopy lie algebra}
For $n\in\bbZ$, an \textbf{$n$-shifted Poisson structure} in a homotopy Lie algebra $\g=(\g,\pi_1)$ is a family of maps 
\begin{equation}
\pi_{\geq2}:=\left\{\pi_w^i:{\Mwedge}^i\g\to(\Sym_\pm^w\g)[(1-w)n+2-i]\right\}_{w\geq2,\,i\in\bbN}
\end{equation}
which satisfies the \textbf{Maurer-Cartan equation}
\begin{equation}\label{eq:MC eqn}
\pb{\pi}{\pi}=0\quad,
\end{equation}
where $\pi:=\pi_1+\pi_{\geq2}$.
\end{defi}
Using the bullet product formula \eqref{eq:definition of bullet product} and expanding the Maurer-Cartan equation \eqref{eq:MC eqn} into its weight and input components, we must have: $\forall w\geq2$ and $\forall i\in\bbN$,
\begin{equation}
\sum_{\begin{smallmatrix}\p+q=w\\ \j+k=i
\end{smallmatrix}}(-1)^{(\p n+j)(\q n+\k)+\j\q n}\Sigma_{p,\q}^\pm[\pi_p^j\otimes1_{\q}][1_{\j}\otimes\pi_q^k]\Sigma_{\j,k}=0\quad.
\end{equation}
\begin{rem}\label{rem:2SPS}
This paper is particularly concerned with 2-shifted Poisson structures, i.e. 
\begin{subequations}\label{subeq:2SPS}
\begin{align}
\left\{\pi_w^i:{\Mwedge}^i\g\to(\Sym^w\g)[4-2w-i]\right\}_{w\geq2,i\in\bbN}
\end{align} 
such that: $\forall w\geq2$ and $\forall i\in\bbN$,  
\begin{equation}\label{eq:2sps mc}
\sum_{\begin{smallmatrix}\p+q=w\\ \j+k=i
\end{smallmatrix}}(-1)^{j\k}\Sigma_{p,\q}^+[\pi_p^j\otimes1_{\q}][1_{\j}\otimes\pi_q^k]\Sigma_{\j,k}=0\quad.
\end{equation}
\end{subequations}
Subsection \ref{subsec:Poisson inf2bra} will make use of the weight terms $w\in\{2,3\}$, i.e.:
\begin{alignat}{6}
&\left\{\pi_2^i:{\Mwedge}^i\g\to(\Sym^2\g)[-i]\right\}_{i\in\bbN}\quad&&,\label{eq:weight 2 of 2SPS}\\
&\left\{\pi_3^i:{\Mwedge}^i\g\to(\Sym^3\g)[-2-i]\right\}_{i\in\bbN}&&,\label{eq:weight 3 of 2SPS}
\end{alignat}
together with the first two weight-relations of the Maurer-Cartan equation \eqref{eq:2sps mc}: $\forall i\in\bbN$,
\begin{align}
\sum_{\j+k=i}(-1)^{j\k}\left(\Sigma^+_{1,1}[\pi_1^j\otimes1_\g][1_{\j}\otimes\pi_2^k]+\pi_2^j[1_{\j}\otimes\pi_1^k]\right)\Sigma_{\j,k}=\,&0\,,\label{eq:weight 2 MC}\\
\sum_{\j+k=i}(-1)^{j\k}\left(\Sigma^+_{1,2}[\pi_1^j\otimes1_{\g^{\otimes2}}][1_{\j}\otimes\pi_3^k]+\pi_3^j[1_{\j}\otimes\pi_1^k]+\Sigma^+_{2,1}[\pi_2^j\otimes1_\g][1_{\j}\otimes\pi_2^k]\right)\Sigma_{\j,k}=\,&0\,.\label{eq:weight 3 MC}
\end{align}
\end{rem}

\begin{ex}
Following Example \ref{ex:DGLA and ordinary Lie algebra}, if $\g=\g^0$ is an ordinary Lie algebra then it is trivial to check that a 2-shifted Poisson structure is merely a symmetric tensor $\pi_2^0:\bbK\to\Sym^2\g$ such that $\Sigma_{1,1}^+[\pi^2_1\otimes\id_\g][\id_\g\otimes\pi_2^0]=0$, i.e. a symmetric ad-invariant tensor $\pi_2^0\in(\Sym^2\g)^{\ad\g}$.
\end{ex}

\subsection{Homotopy 2-categories and infinitesimal 2-braidings}\label{sub:inf2bra}
We denote by
\begin{align}
\Ch^{[-1,0]}\,\subseteq\, \Ch
\end{align}
the full subcategory of cochain complexes $V\in \Ch$
concentrated in degrees $\{-1,0\}$. We shall often denote such objects by
\begin{align}
V = \Big(\xymatrix@C=2em{
V^{-1} \,\ar[r]^-{d} \,&\, V^0
}\Big)
\,\in\, \Ch^{[-1,0]} \quad.
\end{align}
Note that the tensor product $\otimes$ on $\Ch$ does not directly restrict
to the subcategory $\Ch^{[-1,0]}\subseteq\Ch$ because
$V\otimes W\in \Ch$ is, in general, concentrated in degrees $\{-2,-1,0\}$,
for $V,W\in \Ch^{[-1,0]}$. This issue can be resolved by composing the tensor product $\otimes$ with
the \textbf{good truncation functor} $\tau^{[-1,0]}:\Ch\to\Ch^{[-1,0]}$, recall this is defined on objects $U\in\Ch$ simply as
\begin{flalign}
\tau^{[-1,0]}(U)\,:=\, \bigg(
\xymatrix@C=4em{
\frac{U^{-1}}{B^{-1}(U)} \ar[r]^-{d} ~&~
Z^0(U)
}
\bigg)\quad.
\end{flalign}
Remark \ref{rem:cochain map restricts to coboundaries/cocycles} tells us that the morphism map of $\tau^{[-1,0]}:\Ch\to\Ch^{[-1,0]}$ is simply a restriction. The \textbf{truncated tensor product} is the following composition of functors,
\begin{subequations}\label{eqn:truncatedtensor}
\begin{align}
\boxtimes\,:\,
\xymatrix@C=3em{
\Ch^{[-1,0]}\times \Ch^{[-1,0]}\ar[r]^-{\subseteq} ~&~ \Ch\times\Ch\ar[r]^-{\otimes} ~&~
\Ch \ar[r]^-{\tau^{[-1,0]}} ~&~ \Ch^{[-1,0]}
}\quad,
\end{align}
which reads explicitly on objects $V,W\in \Ch^{[-1,0]}$ as
\begin{flalign}
V\boxtimes W\,:=\, \bigg(
\xymatrix@C=4em{
\frac{\big(V^{-1}\otimes W^0\big)\oplus\big(V^0\otimes W^{-1}\big)}{d\big(V^{-1}\otimes W^{-1}\big)} \ar[r]^-{d} ~&~
V^0\otimes W^0
}
\bigg)\quad.
\end{flalign}
\end{subequations}
One directly checks that $\Ch^{[-1,0]}$ forms a symmetric monoidal category
with respect to the truncated tensor product $\boxtimes$, the monoidal unit $\bbK\in\Ch^{[-1,0]}$
and the symmetric braiding
\begin{align}
s_{V,W}:V\boxtimes W\to W\boxtimes V\qquad,\qquad v\otimes w\mapsto w\otimes v\quad,
\end{align}
where all Koszul signs are trivial
because there are no odd-odd degree combinations in the truncated tensor product \eqref{eqn:truncatedtensor}. We observe that the symmetric monoidal categories
$(\Ch,\otimes,\bbK,s)$ and $(\Ch^{[-1,0]},\boxtimes,\bbK,s)$
are related by a lax symmetric monoidal functor.
\begin{lem}\label{lem:laxtruncation}
The truncation functor $\tau^{[-1,0]}:\Ch\to\Ch^{[-1,0]}$ is canonically lax symmetric monoidal.
\end{lem}
\begin{proof}
The functor $\tau^{[-1,0]}$ preserves the monoidal units strictly, i.e.\ $\tau^{[-1,0]}(\bbK)=\bbK$,
and the lax structure $\tau^{[-1,0]}(V)\,\boxtimes\,\tau^{[-1,0]}(W)\to 
\tau^{[-1,0]}\big(V\otimes W\big)$ for the tensor product of $V,W\in\Ch$ 
is given by the obvious cochain map (drawn vertically)
\begin{flalign}
\begin{gathered}
\xymatrix@C=4em@R=2em{
\ar[d] \frac{\big(\tfrac{V^{-1}}{d(V^{-2})}\otimes Z^0(W)\big)\oplus\big(Z^0(V)\otimes\tfrac{W^{-1}}{d(W^{-2})}\big)}{d\big(V^{-1}\otimes W^{-1}\big)} \ar[r]^-{d} ~&~ Z^0(V)\otimes Z^0(W)
\ar[d]\\
\frac{(V\otimes W)^{-1}}{d\big(V\otimes W\big)^{-2}}\ar[r]_-{d}~&~Z^0(V\otimes W)
}
\end{gathered}\quad.
\end{flalign}
\end{proof}
In complete analogy to Subsection \ref{sub:symmetric monoidal dg-cats}, there exists a symmetric strict monoidal 2-category $(\dgCat^{[-1,0]},\smallbox,\K,S)$ which allows us to define the notion of a symmetric strict monoidal $\Ch^{[-1,0]}$-category. The lax symmetric monoidal truncation functor  $\tau^{[-1,0]}:\Ch\to\Ch^{[-1,0]}$ induces a change of base 2-functor
\begin{flalign}\label{eqn:homotopy2category}
\Ho_2:\dgCat\to\dgCat^{[-1,0]}
\end{flalign}
which assigns to a dg-category $\C\in\dgCat$, 
with objects $\C_0$ and hom-complexes $\C[V,W]\in\Ch$, 
its \textbf{homotopy 2-category} $\Ho_2(\C)\in\dgCat^{[-1,0]}$, 
which has the same objects $\C_0$ and hom-complexes
given by the truncated cochain complexes $\tau^{[-1,0]}\big(\C[V,W]\big)\in\Ch^{[-1,0]}$. Given a dg-functor $F:\C\to\D$, the $\Ch^{[-1,0]}$-functor $\Ho_2(F):\Ho_2(\C)\to\Ho_2(\D)$ has the same object map as $F$ and its morphism map is merely a restriction of that of $F$. Given a dg-natural transformation $\xi:F\Rightarrow G:\C\to\D$, the $\Ch^{[-1,0]}$-natural transformation $\Ho_2(\xi):\Ho_2(F)\Rightarrow\Ho_2(G)$ is simply defined by $\Ho_2(\xi):=\xi$. 
\sk

The 2-functor \eqref{eqn:homotopy2category} assigning homotopy $2$-categories
is lax symmetric monoidal with the lax structure 
$\Ho_2(\C)\smallbox\Ho_2(\D)\to\Ho_2(\C\cmp\D)$
induced from the lax structure of the good truncation functor in
Lemma \ref{lem:laxtruncation}, for all $\C,\D\in\dgCat$.
For later use, we record the following direct consequence of this result.
\begin{cor}
Let $(\C,\odot,I,\gamma)$ be a symmetric strict monoidal dg-category, the homotopy 2-category $\Ho_2(\C)\in\dgCat^{[-1,0]}$
is canonically a symmetric strict monoidal $\Ch^{[-1,0]}$-category.
\end{cor}
We now recall an important weakened variant of a $\Ch^{[-1,0]}$-natural transformation; see \cite[Definition 2.1]{Kemp25'} for the following definition.
\begin{defi}\label{def:pseudonatural}
Given $\Ch^{[-1,0]}$-categories $\CC,\CC'$ and $\Ch^{[-1,0]}$-functors $F,F':\CC\to\CC'$, a \textbf{pseudonatural
transformation} $\xi:F\Rightarrow F':\CC\to\CC'$ consists of the following two pieces of data:
\begin{enumerate}
\item[(i)] For each object $U\in \CC$, a degree 0 morphism $\xi_U\in\CC'[F(U),F'(U)]^0$.
\item[(ii)] For each pair of objects $U,U'\in \CC$, a homotopy $\xi_{(\cdot)}:\CC[U,U']\to\CC'\left[F(U),F'(U')\right][-1]$. 
\end{enumerate}
These two pieces of data are required to satisfy the following two axioms: for all $f\in\CC[U,U']$ and $f'\in \CC[U',U'']$,
\begin{subequations}
\begin{alignat}{2}
F'(f)\,\xi_U - \xi_{U'}\,F(f)\,&=\partial\xi_f+\xi_{\partial f}&&,\label{eq:dubindex is homotopy}\\
\xi_{f' f}\,&=\,\,\xi_{f'}\,F(f) + F'(f')\,\xi_f\quad&&.\label{eqn:dubindex splits prods}
\end{alignat}
\end{subequations}
\end{defi}
Notice that a $\Ch^{[-1,0]}$-natural transformation is simply a pseudonatural transformation with trivial homotopy components $\xi_f=0$, e.g. the \textbf{identity pseudonatural transformation}
$\Id_F:F\Rightarrow F$ defined by $(\Id_F)_U:=1_{F(U)}$ and
$(\Id_F)_f:=0$.
\begin{rem}
Given a dg-pseudonatural transformation $\varrho:F\Rightarrow G:\C\to\D$ as in Definition \ref{def:dg-pseudonatural transformation}, we get a pseudonatural transformation $\varrho:F\Rightarrow G:\Ho_2(\C)\to\Ho_2(\D)$.
\end{rem}
Given two pseudonatural transformations $F\xRightarrow{\xi}F'\xRightarrow{\xi'}F'':\CC\to\CC'$, the \textbf{vertical composite}
\begin{subequations}\label{eqn: ver comp of pseudos}
\begin{equation}
\begin{tikzcd}
	{\CC} && {\CC'}
	\arrow[""{name=0, anchor=center, inner sep=0}, "F'"{description}, from=1-1, to=1-3]
	\arrow[""{name=1, anchor=center, inner sep=0}, "F", curve={height=-32pt}, from=1-1, to=1-3]
	\arrow[""{name=2, anchor=center, inner sep=0}, "F''"',curve={height=32pt}, from=1-1, to=1-3]
	\arrow["\xi"', shorten <=5pt, shorten >=5pt, Rightarrow, from=1, to=0]
	\arrow["\xi'"', shorten <=5pt, shorten >=5pt, Rightarrow, from=0, to=2]
\end{tikzcd}~~\stackrel{\circ}{\longmapsto}~~
\begin{tikzcd}
	{\CC} && {\CC'}
	\arrow[""{name=1, anchor=center, inner sep=0}, "F", curve={height=-32pt}, from=1-1, to=1-3]
	\arrow[""{name=2, anchor=center, inner sep=0}, "F''"', curve={height=32pt}, from=1-1, to=1-3]
	\arrow["\xi'\,\circ\,\xi"', shorten <=5pt, shorten >=5pt, Rightarrow, from=1, to=2]
\end{tikzcd}
\end{equation}
is defined by:
\begin{alignat}{2}
(\xi'\circ\xi)_U\,&:=\, \xi'_U\,\xi_U\quad&&,\label{eq:object components of vercomp}\\
(\xi'\circ\xi)_f\,&:=\, \xi'_f\,\xi_U + \xi'_{U'}\,\xi_f&&.\label{eq:hom components of vercomp pseudos}
\end{alignat}
\end{subequations}
Given pseudonatural transformations $F\xRightarrow{\xi}F':\CC\to\CC'$ and $G\xRightarrow{\theta}G':\CC'\to\CC''$, the \textbf{horizontal composite}
\begin{subequations}\label{subeq:horcomp of pseudos}
\begin{equation}
\begin{tikzcd}
	{\CC} && {\CC'} && {\CC''}
	\arrow[""{name=0, anchor=center, inner sep=0}, "F", curve={height=-32pt}, from=1-1, to=1-3]
	\arrow[""{name=1, anchor=center, inner sep=0}, "F'"', curve={height=32pt}, from=1-1, to=1-3]
	\arrow[""{name=2, anchor=center, inner sep=0}, "{G}", curve={height=-32pt}, from=1-3, to=1-5]
	\arrow[""{name=3, anchor=center, inner sep=0}, "{G'}"', curve={height=32pt}, from=1-3, to=1-5]
	\arrow["\xi\,"', shorten <=6pt, shorten >=6pt, Rightarrow, from=0, to=1]
	\arrow["{\theta}"', shorten <=6pt, shorten >=6pt, Rightarrow, from=2, to=3]
\end{tikzcd}
~~\stackrel{\ast}{\longmapsto}~~
\begin{tikzcd}
	{\CC} && {\CC''}
	\arrow[""{name=0, anchor=center, inner sep=0}, "GF", curve={height=-32pt}, from=1-1, to=1-3]
	\arrow[""{name=1, anchor=center, inner sep=0}, "G'F'"', curve={height=32pt}, from=1-1, to=1-3]
	\arrow["\theta\,\ast\,\xi\,"', shorten <=6pt, shorten >=6pt, Rightarrow, from=0, to=1]
\end{tikzcd}
\end{equation}
is defined by:
\begin{alignat}{2}
(\theta\ast\xi)_U\,&:=\,\theta_{F'(U)}\,G(\xi_U)\quad&&,\label{eq:object components of horcomp}\\
(\theta\ast\xi)_f\,&:=\,\theta_{F'(f)}\,G(\xi_U)
+\theta_{F'(U')}\,G(\xi_f)&&.\label{eq:homotoper components of horcomp pseudos}
\end{alignat}
\end{subequations}
Given pseudonatural transformations $F\xRightarrow{\xi}F':\CC\to\CC'$ and $G\xRightarrow{\theta}G':\DD\to\DD'$, the \textbf{monoidal composite}
\begin{subequations}\label{subeq:moncomp of pseudos}
\begin{equation}
\begin{tikzcd}
	{\CC} && {\CC'}
	\arrow[""{name=0, anchor=center, inner sep=0}, "F", curve={height=-32pt}, from=1-1, to=1-3]
	\arrow[""{name=1, anchor=center, inner sep=0}, "{F'}"', curve={height=32pt}, from=1-1, to=1-3]
	\arrow["\xi\,"', shorten <=6pt, shorten >=6pt, Rightarrow, from=0, to=1]
\end{tikzcd}
~~
\begin{tikzcd}
	{\DD} && {\DD'}
	\arrow[""{name=0, anchor=center, inner sep=0}, "G", curve={height=-32pt}, from=1-1, to=1-3]
	\arrow[""{name=1, anchor=center, inner sep=0}, "{G'}"', curve={height=32pt}, from=1-1, to=1-3]
	\arrow["\theta\,"', shorten <=6pt, shorten >=6pt, Rightarrow, from=0, to=1]
\end{tikzcd}
~~\stackrel{\smallbox}{\longmapsto}~~
\begin{tikzcd}
	{\CC\smallbox\DD} && {\CC'\smallbox\DD'}
	\arrow[""{name=0, anchor=center, inner sep=0}, "F\smallbox G", curve={height=-32pt}, from=1-1, to=1-3]
	\arrow[""{name=1, anchor=center, inner sep=0}, "{F'\smallbox G'}"', curve={height=32pt}, from=1-1, to=1-3]
	\arrow["\xi\smallbox\theta\,"', shorten <=6pt, shorten >=6pt, Rightarrow, from=0, to=1]
\end{tikzcd}
\end{equation}
is defined by:
\begin{alignat}{2}
(\xi\smallbox\theta)_{U,V}\,&:=\,\xi_U\boxtimes\theta_V\quad&&,\label{eq:object components of moncomp}\\
(\xi\smallbox\theta)_{f,g}\,&:=\,
\xi_f\boxtimes G'(g)\theta_V + \xi_{U'}F(f)\boxtimes \theta_g&&.\label{eq:homotoper components of moncomp pseudos}
\end{alignat}
\end{subequations}
Note that pseudonaturality \eqref{eq:dubindex is homotopy} together with the truncation allows us to equivalently write 
\begin{equation}\label{eq:alternative moncomp of pseudos}
(\xi\smallbox\theta)_{f,g}\,=\,F'(f)\xi_U\boxtimes\theta_g+
\xi_f\boxtimes\theta_{V'}G(g)\quad.
\end{equation}
\begin{defi}\label{def: strict t}
Given a symmetric strict monoidal $\Ch^{[-1,0]}$-category $(\CC,\odot,I,\gamma)$, we say a pseudonatural transformation $t:\odot\Rightarrow\odot:\CC\smallbox\CC\to\CC$ is an \textbf{infinitesimal 2-braiding} if it satisfies the \textbf{left infinitesimal hexagon relation}:
\begin{subequations}\label{subeq:t_U(VW)}
\begin{alignat}{6}
t_{U,V\odot W}=\,&t_{U,V}\odot1_W+(1_U\odot \gamma_{W,V})(t_{U,W}\odot1_V)(1_U\odot \gamma_{V,W})\quad&&,\\
t_{f,g\odot h}=\,&t_{f,g}\odot h+(1_{U'}\odot \gamma_{W',V'})(t_{f,h}\odot g)(1_U\odot \gamma_{V,W})&&,
\end{alignat}
\end{subequations}
and the \textbf{right infinitesimal hexagon relation}:
\begin{subequations}\label{subeq:t_(UV)W}
\begin{alignat}{6}
t_{U\odot V,W}=\,&1_U\odot t_{V,W}+(\gamma_{V,U}\odot1_W)(1_V\odot t_{U,W})(\gamma_{U,V}\odot1_W)\quad&&,\\
t_{f\odot g,h}=\,&f\odot t_{g,h}+(\gamma_{V',U'}\odot1_{W'})(g\odot t_{f,h})(\gamma_{U,V}\odot1_W)&&.
\end{alignat}
\end{subequations}
\end{defi}

\begin{defi}\label{def:sym t}
We say a pseudonatural transformation $t:\odot\Rightarrow\odot:\CC\smallbox\CC\to\CC$ is \textbf{$\gamma$-equivariant} if it intertwines with the symmetric braiding $\gamma$,
\begin{subequations}
\begin{alignat}{2}
\label{intertwine single-index}\gamma_{U,V}\,t_{U,V}=\,&t_{V,U}\,\gamma_{U,V}\quad&&,\\\label{intertwine homotopy}
\gamma_{U',V'}\,t_{f,g}=\,&t_{g,f}\,\gamma_{U,V}&&.
\end{alignat}
\end{subequations}
\end{defi}
\begin{lem}\label{lem:sym t then half conditions}
If $t:\odot\Rightarrow\odot:\CC\smallbox\CC\to\CC$ is $\gamma$-equivariant then \eqref{subeq:t_U(VW)}$\iff$\eqref{subeq:t_(UV)W}.
\end{lem}

We recall from \cite[Definitions 17 and 18]{Joao}, Cirio and Faria Martins'\footnote{Note that, throughout their joint publications, compositions of morphisms are written from left to right.} important notions of coherency and total $\gamma$-equivariance, respectively. Actually, we specifically use the simplifications provided in \cite[Subsection 5.2]{Kemp25}. 
\begin{defi}\label{def: coherency}
A $\gamma$-equivariant infinitesimal 2-braiding is \textbf{coherent} if it satisfies
\begin{align}
t_{t_{U,V},1_W}+t_{1_U,t_{V,W}}+(1_U\odot\gamma_{W,V})t_{t_{U,W},1_V}(1_U\odot\gamma_{V,W})=0\quad.
\end{align}
A $\gamma$-equivariant infinitesimal 2-braiding is \textbf{totally $\gamma$-equivariant} if it satisfies
\begin{align}\label{left total symmetry}
t_{\gamma_{U,V},1_W}=0\quad.
\end{align}
\end{defi}

\newpage
\section{Representations of homotopy Lie algebras}\label{sec:weak reps}
Throughout this section, we fix a homotopy Lie algebra $\g$ as in Definition \ref{defi: L infinity algebra}. The aim of this section is to show that there is a symmetric monoidal dg-category (see Definition \ref{def:symmetric strict monoidal dg-category}) of representations $\Rep$ and that, given a 2-shifted Poisson structure as in Remark \ref{rem:2SPS}, one can construct a totally $\gamma$-equivariant \textit{coherent} infinitesimal 2-braiding (see Definition \ref{def: coherency}) on the homotopy 2-category $\Ho_2[\Rep]$. 
\sk

Subsection \ref{sub:intertwiners and higher morphisms} begins by first defining intertwiners between graded vector spaces then gives five different constructions of intertwiners that will be commonly used. For example, the juxtaposition of intertwiners \eqref{subeq:juxtaposition of intertwiners} will be shown to be associative and unital in Lemma \ref{lem:juxtaposition is assoc/unital}. Proposition \ref{propo:Rep is a dg-cat} defines the dg-category $\Rep$ whereas Proposition \ref{propo: sym dg} constructs the compatible symmetric monoidal structure. Subsection \ref{subsec:Poisson inf2bra} provides the representation-theoretic analogue to the construction of infinitesimal 2-braidings from 2-shifted Poisson structures in \cite[Subsection 3.2]{KLS25}; \textit{the main result of this paper is Theorem \ref{theo:2SPS induces coherent inf2bra}} where we show that such infinitesimal 2-braidings are, in particular, coherent. 
\sk

\subsection{Intertwiners, their composition and monoidal product}\label{sub:intertwiners and higher morphisms}
\begin{defi}\label{def:squiggly notation}
Given graded vector spaces $U$ and $V$ together with an integer $n\in\bbZ$, an \textbf{intertwiner} $f:U\s V[n]$ is a family of fully skew-symmetric graded linear maps:
\begin{align}\label{eq:def of intertwiner}
\Big\{f^i:\g^{\otimes\i}\otimes U\to V\big[n-\i\big]\Big\}_{i\geq1}\quad.
\end{align}
\end{defi}
For clarity, we emphasise that \textit{full} skew-symmetry of \eqref{eq:def of intertwiner} means in \textit{all} entries, e.g.
\begin{align}
f^i(x_1,x_2,x_3,\ldots,x_{\i-1},x_{\i},u)=&-(-1)^{|x_1||x_2|}f^i(x_2,x_1,x_3,\ldots,x_{\i-1},x_{\i},u)\nn\\
=&-(-1)^{|x_{\i}||u|}f^i(x_1,x_2,x_3,\ldots,x_{\i-1},u,x_{\i})\quad.
\end{align}
Analogous to Remark \ref{rem:tautological rewrite of graded linear map}, we rewrite $f:U\s V[n]$ as $f:U\s V\big[|f|\big]$; in the case that $|f|=0$, we simply write $f:U\s V$.
\begin{ex}
Definition \ref{defi: L infinity algebra} of a homotopy Lie algebra is the special case $U=V=\g$ and $f=\pi_1:\g\s\g[1]$ thus $|\pi_1|=1$, whereas the generalised Jacobi identity is the requirement $\pb{\pi_1}{\pi_1}=0:\g\s\g[2]$.
\end{ex}
Let us now provide a list of constructions of intertwiners that will be fundamental throughout:
\begin{enumerate}
\item[(i)] Given a graded vector space $U$, we define the \textbf{identity intertwiner} simply as
\begin{align}\label{eq:unit intertwiner}
\oone_U:=\left\{1_U:U\to U\right\}:U\s U\quad.
\end{align}
\item[(ii)] Given a graded vector space $U$, for every $i\geq1$ we uniquely extend 
\begin{equation}\label{eq:pi at U}
\ell^{\i}\otimes1_U:\g^{\otimes\i}\otimes U\to(\g\otimes U)\big[2-\i\big]
\end{equation}
to be fully skew-symmetric and denote the resulting intertwiner as $\ell_U:U\s(\g\otimes U)[2]$.
\item[(iii)] Given $f:U\s V\big[|f|\big]$, we define $\varrho_f:\g\otimes U\s V\big[|f|-1\big]$ as
\begin{align}\label{eq:homotoper component of varrho psnat}
\varrho_f^i:=(-1)^if^{i+1}:\g^{\otimes i}\otimes U\to V\big[|f|-i\big]\quad.
\end{align}
\item[(iv)] Given $f:U\s V\big[|f|\big]$ and $g:V\s W\big[|g|\big]$, we define their \textbf{juxtaposition}
\begin{subequations}\label{subeq:juxtaposition of intertwiners}
\begin{equation}
gf:U\s W\big[|g|+|f|\big]
\end{equation}
by setting its components as
\begin{align}
(gf)^i:=\sum_{\j+k=i}(-1)^{(|g|-\j)\k}g^j[1_{\j}\otimes f^k][\Sigma_{\j,\k}\otimes1_U]:\g^{\otimes\i}\otimes U\to W\big[|g|+|f|-\i\big]\quad.
\end{align}
\end{subequations}
\item[(v)] Given $f:U\s U'\big[|f|\big]$ and $g:V\s V'\big[|g|\big]$, we define their \textbf{odot product}
\begin{subequations}\label{eq:odot of squiggly maps}
\begin{equation}
f\odot g:U\otimes V\s(U'\otimes V')\big[|f|+|g|\big]
\end{equation}
by setting its components as
\begin{align}
(f\odot g)^i:=\sum_{\j+k=i}(-1)^{(|f|-\j)\k}[f^j\otimes g^k][1_{\j}\otimes s_{\k,U}\otimes1_V][\Sigma_{\j,\k}\otimes1_{U\otimes V}]\quad.
\end{align}
\end{subequations} 
\end{enumerate}
\begin{lem}\label{lem:juxtaposition is assoc/unital}
Juxtaposition \eqref{subeq:juxtaposition of intertwiners} is associative and unital with respect to the identities \eqref{eq:unit intertwiner}. 
\end{lem}
\begin{proof}
Suppose we are given $f:U\s V\big[|f|\big]$, $g:V\s W\big[|g|\big]$ and $h:W\s X\big[|h|\big]$ then
\begin{align}
[h(gf)]^i\overset{\eqref{subeq:juxtaposition of intertwiners}}{=}&\sum_{\j+k=i}\,\sum_{\l+m=k}(-1)^{(|h|-\j)\k+(|g|-\l)\m}h^j[1_{\j}\otimes g^l(1_{\l}\otimes f^m)(\Sigma_{\l,\m}\otimes1_U)][\Sigma_{\j,\k}\otimes1_U]\nn\\
\overset{\eqref{eq:associativity of shuffler}}{=}&\sum_{\j+\l+m=i}(-1)^{(|h|-\j)(l+m)+(|g|-\l)\m}h^j[1_{\j}\otimes g^l(1_{\l}\otimes f^m)][\Sigma_{\j,\l,\m}\otimes1_U]\nn\\
\overset{\eqref{eq:associativity of shuffler}}{=}&\sum_{\k+m=i}\,\sum_{\j+l=k}(-1)^{(|h|+|g|-\k)\m+(|h|-\j)\l}h^j[1_{\j}\otimes g^l][\Sigma_{\j,\l}\otimes1_V][1_{\k}\otimes f^m][\Sigma_{\k,\m}\otimes1_U]\nn\\\overset{\eqref{subeq:juxtaposition of intertwiners}}{=}&[(hg)f]^i\quad.\label{eq:juxtaposition is associative}
\end{align}
Unitality is obvious.
\end{proof}
\begin{lem}\label{lem:intertwining with pi maps}
Every intertwiner $f:U\s V\big[|f|\big]$ intertwines with the ell maps \eqref{eq:pi at U}, i.e. 
\begin{align}\label{eq:pi_Vf=(1xf)pi_U}
\ell_Vf=(\oone_\g\odot f)\ell_U\quad.
\end{align}
\end{lem}
\begin{proof}
First note that
\begin{equation}\label{eq:(oone_g odot f)^i}
(\oone_\g\odot f)^i\overset{\eqref{eq:odot of squiggly maps},\eqref{eq:unit intertwiner}}{=}(1_\g\otimes f^i)(s_{\i,\g}\otimes1_U)
\end{equation}
thus
\begin{align*}
[(\oone_\g\odot f)\ell_U]^i\overset{\eqref{subeq:juxtaposition of intertwiners}}{=}\qquad&\sum_{\j+k=i}(-1)^{(|f|-\j)\k}[1_\g\otimes f^j][s_{\j,\g}\otimes1_U][1_{\j}\otimes\ell^{\k}\otimes1_U][\Sigma_{\j,\k}\otimes1_U]\\
\overset{\eqref{eq:naturality of symmetric braiding},\eqref{eq:symmetry of shuffler},\eqref{eq:interchange for tensor product}}{=}&\sum_{j+\k=i}(-1)^{\j\k}[\ell^{\k}\otimes f^j][\Sigma_{\k,\j}\otimes1_U]\\
\overset{\eqref{subeq:juxtaposition of intertwiners}}{=}\qquad&(\ell_Vf)^i\quad.\numberthis
\end{align*}
\end{proof}
\begin{lem}\label{lem:varrho map is a derivation}
Given $f:U\s V\big[|f|\big]$ and $g:V\s W\big[|g|\big]$, the varrho map \eqref{eq:homotoper component of varrho psnat} satisfies
\begin{align}\label{eq:varrho is associative}
\varrho_{gf}=\varrho_g(\oone_\g\odot f)+(-1)^{|g|}g\varrho_f\quad.
\end{align}
\end{lem}
\begin{proof}
Recalling the components \eqref{eq:(oone_g odot f)^i} of $\oone_\g\odot f$, we have
\begin{align*}
\varrho_{gf}^i\overset{\eqref{subeq:juxtaposition of intertwiners}}{=}&(-1)^i\sum_{\j+\k=i}(-1)^{(|g|-\j)\k}g^j[1_{\j}\otimes f^k][\Sigma_{\j,\k}\otimes1_U]\\
\overset{\eqref{eq:unitality of shuffler}}{=}&\sum_{\j+\k=i}(-1)^{\j+(|g|-\j)\k}g^j[1_{\j}\otimes f^k][1_{\j-1}\otimes s_{\k,\g}\otimes1_U][\Sigma_{\j-1,\k}\otimes1_{\g U}]\\&\qquad+(-1)^{\j+(|g|-j)\k}g^j[1_{\j}\otimes f^k][\Sigma_{\j,\k-1}\otimes1_{\g U}]\\
\overset{\eqref{subeq:juxtaposition of intertwiners}}{=}&\big[\varrho_g(\oone_\g\odot f)+(-1)^{|g|}g\varrho_f\big]^i\quad.\numberthis
\end{align*}
\end{proof}
\begin{cor}\label{cor:[rho,cdot] is derivation}
Given $f:U\s V\big[|f|\big]$, $\rho_U:U\s U[1]$ and $\rho_V:V\s V[1]$, if we define
\begin{align}\label{eq:diff of squiggly map}
\pb{\rho}{f}:=\rho_Vf-(-1)^{|f|}f\rho_U-\varrho_f\ell_U:U\s V\big[|f|+1\big]
\end{align}
then the map $\pb{\rho}{\cdot\,}$ is a derivation of the juxtaposition \eqref{subeq:juxtaposition of intertwiners}, i.e. given also $g:V\s W\big[|g|\big]$ and $\rho_W:W\s W[1]$,
\begin{equation}\label{eq:[rho,gf]=...}
\pb{\rho}{gf}=\pb{\rho}{g}f+(-1)^{|g|}g\pb{\rho}{f}\quad.
\end{equation}
\end{cor}
\begin{proof}
Let us expand the RHS of \eqref{eq:[rho,gf]=...},
\begin{align*}
\pb{\rho}{g}f+(-1)^{|g|}g\pb{\rho}{f}\overset{\eqref{eq:diff of squiggly map}}{=}\big(&\rho_Wg-(-1)^{|g|}g\rho_V-\varrho_g\ell_V\big)f+(-1)^{|g|}g\big(\rho_Vf-(-1)^{|f|}f\rho_U-\varrho_f\ell_U\big)\\
\overset{\eqref{eq:juxtaposition is associative},\eqref{eq:pi_Vf=(1xf)pi_U}}{=}&\rho_Wgf-(-1)^{|gf|}gf\rho_U-\varrho_g(\oone_\g\odot f)\ell_U-(-1)^{|g|}g\varrho_f\ell_U\\
\overset{\eqref{eq:varrho is associative}}{=}\,\,~&\rho_Wgf-(-1)^{|gf|}gf\rho_U-\varrho_{gf}\ell_U\\
\overset{\eqref{eq:diff of squiggly map}}{=}\,\,~&\pb{\rho}{gf}\quad.\numberthis\label{eq:comp of squig maps is cochain}
\end{align*}
\end{proof}
\begin{defi}\label{defi: weak rep}
Within the context of Corollary \ref{cor:[rho,cdot] is derivation}:
\begin{enumerate}
\item[(i)] If $|f|=0$ then we say that the intertwiner $f:U\s V$ is an \textbf{equivariant map} if 
\begin{equation}
\pb{\rho}{f}=0\quad,
\end{equation}
i.e.
\begin{equation}\label{eq:equivariant as juxtaposition}
f\rho_U+\varrho_f\ell_U=\rho_Vf\quad.
\end{equation}
\item[(ii)] We say $V=(V,\rho_V)$ is a \textbf{representation} of $\g$ if $\rho_V:V\s V[1]$ satisfies
\begin{equation}\label{eq:[rho,rho_V]}
\pb{\rho}{\rho_V}=\rho_V\rho_V\quad,
\end{equation}
i.e. 
\begin{equation}\label{eq:rho_Vrho_V=..}
\rho_V\rho_V=\varrho_{\rho_V}\ell_V\quad.
\end{equation}
\end{enumerate}
\end{defi}
The following remark shows that Definition \ref{defi: weak rep} is equivalent to \cite[Definitions 2.2 and 3.3]{Allocca}.
\begin{rem}\label{rem:comparison to Allocca}
By using Definition \ref{def:squiggly notation}, the intertwiner $\rho_V:V\s V[1]$ is a family of fully skew-symmetric graded linear maps
\begin{subequations}\label{subeq:Allocca rep}
\begin{align}\label{maps of weak rep}
\left\{\rho_V^i:\g^{\otimes\i}\otimes V\to V[2-i]\right\}_{i\geq1}
\end{align}
whereas the fact that $V$ is a representation \eqref{eq:rho_Vrho_V=..} means we must have, $\forall i\geq1$,
\begin{align}\label{eq:rewritten weak action property}
\sum_{\j+k=i}(-1)^{\j\k}\rho^j_V(1_{\j-1}\otimes\pi_1^k\otimes1_V)(\Sigma_{\j-1,k}\otimes1_V)+(-1)^{j\k}\rho^j_V(1_{\j}\otimes\rho^k_V)(\Sigma_{\j,\k}\otimes1_V)=0\quad.
\end{align}
\end{subequations}
Given $x_1,\ldots,x_k\in\g$, we can set $\rho_V^k(x_1,\ldots,x_k):=\pi_1^k(x_1,\ldots,x_k)$ and thus use \eqref{eq:unitality of shuffler} to rewrite \eqref{eq:rewritten weak action property} as 
\begin{align}\label{generalised action property}
\forall i\geq1\,,\qquad\sum_{\j+k=i}(-1)^{j\k}\rho_V^j(1_{\j}\otimes\rho_V^k)\Sigma_{\j,k}=0\quad.
\end{align}
In fact, we can further compactify \eqref{generalised action property} as $\pb{\rho_V}{\rho_V}=0$ where, this time, the bracket is genuinely the Schouten-Nijenhuis bracket \eqref{eq:def of Schouten by bullet}. Likewise, for $\rho_U\neq f\neq\rho_V$, we can set $f^k(x_1,\ldots,x_k):=0$ so that juxtaposition \eqref{subeq:juxtaposition of intertwiners} can be rewritten as
\begin{align}
(gf)^i=\sum_{\j+k=i}(-1)^{(|g|-\j)\k}g^j[1_{\j}\otimes f^k]\Sigma_{\j,k}\overset{\eqref{eq:bullet composition operation}}{=}\sum_{\j+k=i}g^j\bullet f^k\quad,
\end{align}
i.e. we actually have $gf=g\bullet f$. Putting both ideas together, we can rewrite
\begin{align}\label{eq:arity component of [rho,f]}
\pb{\rho}{f}^i\overset{\eqref{eq:diff of squiggly map}}{=}\sum_{\j+k=i}&(-1)^{j\k}\rho_V^j[1_{\j}\otimes f^k][\Sigma_{\j,\k}\otimes1_U]-(-1)^{|f|}(-1)^{(|f|-\j)\k}f^j(1_{\j}\otimes\rho_U^k)(\Sigma_{\j,\k}\otimes1_U)\nn\\&-(-1)^{|f|}(-1)^{(|f|-j)\k}f^j(1_{\j-1}\otimes\pi_1^k\otimes1_U)(\Sigma_{\j-1,k}\otimes1_U)\\\overset{\eqref{eq:unitality of shuffler}}{=}\sum_{\j+k=i}&(-1)^{j\k}\rho_V^j[1_{\j}\otimes f^k]\Sigma_{\j,k}-(-1)^{|f|}(-1)^{(|f|-\j)\k}f^j(1_{\j}\otimes\rho_U^k)\Sigma_{\j,k}\nn
\end{align}
as 
\begin{equation}
\pb{\rho}{f}\overset{\eqref{eq:bullet composition operation}}{=}\rho_V\bullet f-(-1)^{|f|}f\bullet\rho_U
\end{equation}
where, once again, the bracket is genuinely the Schouten-Nijenhuis bracket \eqref{eq:def of Schouten by bullet}. Finally, we say that
\begin{align}\label{morph between weak rep}
f=\left\{f^i:\g^{\otimes\i}\otimes U\to V[1-i]\right\}_{i\geq1}:U\s V
\end{align}
is an equivariant map \eqref{eq:equivariant as juxtaposition} if $\rho_V\bullet f=f\bullet\rho_U$.
\end{rem}
\begin{ex}\label{ex:adjoint representation}
\eqref{generalised action property} makes it obvious that any homotopy Lie algebra $(\g,\pi_1)$ forms a representation of itself with the action maps given by $\rho_\g^i=\pi^i_1$, i.e. the \textbf{adjoint representation}.
\end{ex}
Defining
\begin{equation}\label{eq:def of varrho^i_V}
\rho^i_V:=(-1)^i\varrho^{\i}_V
\end{equation}
then, analogously to Remark \ref{rem:low arity data of L_inf-algebra}, $\varrho_V^0:V\to V[1]$ is a differential structure on $V$ whereas the cochain map $\varrho_V^1:\g\otimes V\to V$ satisfies the usual action property up to the internal hom differential of the \textbf{Jacobiactor} $\varrho_V^2:\g^{\otimes2}\otimes V\to V[-1]$.
\begin{lem}\label{lem:[rho,cdot] squares to zero}
Within the context of Corollary \ref{cor:[rho,cdot] is derivation}, if $U=(U,\rho_U)$ and $V=(V,\rho_V)$ are representations as in Definition \ref{defi: weak rep} then
\begin{equation}
\pb{\rho}{\pb{\rho}{f}}=0\quad.
\end{equation}
\end{lem}
\begin{proof}
We make use of Lemmas \ref{lem:juxtaposition is assoc/unital}, \ref{lem:intertwining with pi maps} and \ref{lem:varrho map is a derivation},
\begin{align*}
\pb{\rho}{\pb{\rho}{f}}\overset{\eqref{eq:diff of squiggly map}}{=}&\pb{\rho}{\rho_Vf}-(-1)^{|f|}\pb{\rho}{f\rho_U}-\pb{\rho}{\varrho_f\ell_U}\\
\overset{\eqref{eq:diff of squiggly map}}{=}&\rho_V\rho_Vf+(-1)^{|f|}\rho_Vf\rho_U-\varrho_{\rho_Vf}\ell_U-(-1)^{|f|}\rho_Vf\rho_U-f\rho_U\rho_U+(-1)^{|f|}\varrho_{f\rho_U}\ell_U\\&-\rho_V\varrho_f\ell_U-(-1)^{|f|}\varrho_f\ell_U\rho_U+\varrho_{\varrho_f\ell_U}\ell_U\\
\overset{\eqref{eq:varrho is associative}}{=}&\rho_V\rho_Vf-\varrho_{\rho_V}(\oone_\g\odot f)\ell_U+\rho_V\varrho_f\ell_U-f\rho_U\rho_U+(-1)^{|f|}\varrho_f(\oone_\g\odot\rho_U)\ell_U+f\varrho_{\rho_U}\ell_U\\&-\rho_V\varrho_f\ell_U-(-1)^{|f|}\varrho_f\ell_U\rho_U+\varrho_{\varrho_f}(\oone_\g\odot\ell_U)\ell_U-(-1)^{|f|}\varrho_f\varrho_{\ell_U}\ell_U\\
\overset{\eqref{eq:rho_Vrho_V=..}}{=}&\varrho_{\rho_V}\ell_Vf-\varrho_{\rho_V}(\oone_\g\odot f)\ell_U+(-1)^{|f|}\varrho_f(\oone_\g\odot\rho_U)\ell_U-(-1)^{|f|}\varrho_f\ell_U\rho_U\\&+\varrho_{\varrho_f}(\oone_\g\odot\ell_U)\ell_U-(-1)^{|f|}\varrho_f\varrho_{\ell_U}\ell_U\\
\overset{\eqref{eq:pi_Vf=(1xf)pi_U}}{=}\,&\varrho_{\varrho_f}(\oone_\g\odot\ell_U)\ell_U-(-1)^{|f|}\varrho_f\varrho_{\ell_U}\ell_U\numberthis\label{eq:proof that [rho,] squares to zero}
\end{align*}
however we already have $(\oone_\g\odot\ell_U)\ell_U=0$ and $\varrho_{\ell_U}\ell_U=0$ because these are simply different ways of rewriting the generalised Jacobi identity \eqref{eq:generalised Jacobi identity}.
\end{proof}
\begin{propo}\label{propo:Rep is a dg-cat}
There is a dg-category $\Rep$ with:
\begin{enumerate}
\item[(i)] Objects given by representations $V=(V,\rho_V)$ as in \eqref{subeq:Allocca rep}.
\item[(ii)] Morphisms given by intertwiners as in Definition \ref{def:squiggly notation}.
\item[(iii)] Composition given by juxtaposition \eqref{subeq:juxtaposition of intertwiners} and hom differential given by \eqref{eq:diff of squiggly map}.
\item[(iv)] Units given by the identity intertwiners \eqref{eq:unit intertwiner}.
\end{enumerate}
\end{propo}
\begin{proof}
The associativity and unitality of composition was proven in Lemma \ref{lem:juxtaposition is assoc/unital}, the map $\pb{\rho}{\cdot\,}$ was shown to satisfy the square-zero condition in Lemma \ref{lem:[rho,cdot] squares to zero} and Corollary \ref{cor:[rho,cdot] is derivation} demonstrated that the composition is a cochain map. Finally, the identity intertwiners \eqref{eq:unit intertwiner} are evidently equivariant maps as in \eqref{eq:equivariant as juxtaposition} because we obviously have
\begin{equation}\label{eq:varrho kills units}
\varrho_{\oone_V}=0\quad.
\end{equation}
\end{proof}
\begin{lem}\label{lem:gamma is natural}
Given graded vector spaces $U$ and $V$, if we define the intertwiner
\begin{align}\label{eq:symmetric braiding of dgRep}
\gamma_{U,V}:=\{s_{U,V}:U\otimes V\to V\otimes U\}:U\otimes V\s V\otimes U
\end{align}
then the following properties hold:
\begin{enumerate}
\item[(i)] The family of such intertwiners is natural, i.e. given $f:U\s U'\big[|f|\big]$ and $g:V\s V'\big[|g|\big]$,
\begin{equation}\label{eq:gamma's naturality}
\gamma_{U',V'}(f\odot g)=(-1)^{|f||g|}(g\odot f)\gamma_{U,V}\quad.
\end{equation}
\item[(ii)] Each $\gamma_{U,V}$ is involutive, i.e.
\begin{equation}\label{eq:gamma is involutive}
\gamma_{V,U}\gamma_{U,V}=\oone_{U\otimes V}\quad.
\end{equation}
\end{enumerate}
\end{lem}
\begin{proof}
\begin{enumerate}
\item[(i)] We have 
\begin{align*}
[\gamma_{U',V'}(f\odot g)]^i\overset{\eqref{subeq:juxtaposition of intertwiners}}{=}&s_{U',V'}(f\odot g)^i\\
\overset{\eqref{eq:odot of squiggly maps}}{=}&\sum_{\j+k=i}(-1)^{(|f|-\j)\k}s_{U',V'}[f^j\otimes g^k][1_{\j}\otimes s_{\k,U}\otimes1_V][\Sigma_{\j,\k}\otimes1_{U\otimes V}]\\
\overset{\eqref{eq:naturality of symmetric braiding},\eqref{eq:symmetry of shuffler}}{=}&\sum_{\j+k=i}(-1)^{(|f|-\j)|g|+\j\k}[g^k\otimes f^j]s_{\j U,\k V}[1_{\j}\otimes s_{\k,U}\otimes1_V][s_{\k,\j}\Sigma_{\k,\j}\otimes1_{U\otimes V}]\\
\overset{\eqref{eq:hexagon axiom of a symmetric braiding}}{=}&(-1)^{|f||g|}\sum_{\j+k=i}(-1)^{(|g|-\k)\j}[g^k\otimes f^j][1_{\k}\otimes s_{\j,V}\otimes1_U][\Sigma_{\k,\j}\otimes s_{U,V}]\\\overset{\eqref{eq:odot of squiggly maps}}{=}&(-1)^{|f||g|}(g\odot f)^i[1_{\i}\otimes s_{U,V}]\\
\overset{\eqref{subeq:juxtaposition of intertwiners}}{=}&(-1)^{|f||g|}[(g\odot f)\gamma_{U,V}]^i\numberthis
\end{align*}
\item[(ii)] Involutivity \eqref{eq:gamma is involutive} is obvious.
\end{enumerate}
\end{proof}
\begin{cor}\label{cor:odot is functorial}
The odot product \eqref{eq:odot of squiggly maps} satisfies the following two properties:
\begin{enumerate}
\item[(i)] Middle-four interchange law, i.e. given also $f':U'\s U''\big[|f'|\big]$ and $g':V'\s V''\big[|g'|\big]$,
\begin{equation}\label{eq:functoriality of odot}
(f'\odot g')(f\odot g)=(-1)^{|g'||f|}f'f\odot g'g\quad.
\end{equation}
\item[(ii)] Associativity, i.e. if we are also given $h:W\s W'\big[|h|\big]$ then we have
\begin{equation}\label{eq:odot is associative}
(f\odot g)\odot h=f\odot(g\odot h)\quad.
\end{equation}
\end{enumerate}
\end{cor}
\begin{proof}
\begin{enumerate}
\item[(i)] We obviously have:
\begin{subequations}
\begin{alignat}{6}
(f\odot\oone_{V'})(\oone_U\odot g)=&\,f\odot g\quad&&,\label{eq:(fx1)(1xg)=}\\
(f'\odot\oone_V)(f\odot\oone_V)=&\,f'f\odot\oone_V&&.\label{eq:(f'x1)(fx1)=}
\end{alignat}
\end{subequations}
This gives us
\begin{align*}
(f'\odot\oone_{V'})(f\odot g)\overset{\eqref{eq:(fx1)(1xg)=}}{=}&(f'\odot\oone_{V'})(f\odot\oone_{V'})(\oone_U\odot g)\\
\overset{\eqref{eq:(f'x1)(fx1)=}}{=}&(f'f\odot\oone_{V'})(\oone_U\odot g)\\
\overset{\eqref{eq:(fx1)(1xg)=}}{=}&f'f\odot g\quad.\numberthis\label{eq:(f'x1)(fxg)}
\end{align*}
We use the naturality and invertibility of $\gamma$ to derive from \eqref{eq:(f'x1)(fxg)} the relation
\begin{equation}\label{eq:(1xg')(fxg)}
(\oone_{U'}\odot g')(f\odot g)=(-1)^{|g'||f|}f\odot g'g
\end{equation}
and now we can say
\begin{align*}
(f'\odot g')(f\odot g)&\overset{\eqref{eq:(fx1)(1xg)=}}{=}(f'\odot\oone_{V''})(\oone_{U'}\odot g')(f\odot g)\\
&\,\overset{\eqref{eq:(1xg')(fxg)}}{=}(-1)^{|g'||f|}(f'\odot\oone_{V''})(f\odot g'g)\\
&\,\overset{\eqref{eq:(f'x1)(fxg)}}{=}(-1)^{|g'||f|}f'f\odot g'g\quad.\numberthis
\end{align*}
\item[(ii)] We obviously have:
\begin{subequations}\label{subeq:odot is associative for units}
\begin{alignat}{6}
(f\odot\oone_V)\odot\oone_W=&\,f\odot(\oone_V\odot\oone_W)\quad&&,\\
(\oone_U\odot g)\odot\oone_W=&\,\oone_U\odot(g\odot\oone_W)&&,\\
(\oone_U\odot\oone_V)\odot h=&\,\oone_U\odot(\oone_V\odot h)&&
\end{alignat}
\end{subequations}
thus
\begin{align*}
(f\odot g)\odot h\overset{\eqref{eq:functoriality of odot}}{=}&([f\odot\oone_{V'}]\odot\oone_{W'})([\oone_U\odot g]\odot\oone_{W'})([\oone_U\odot\oone_V]\odot h)\\
\overset{\eqref{subeq:odot is associative for units}}{=}&(f\odot[\oone_{V'}\odot\oone_{W'}])(\oone_U\odot[g\odot\oone_{W'}])(\oone_U\odot[\oone_V\odot h])\\
\overset{\eqref{eq:functoriality of odot}}{=}&f\odot(g\odot h)\quad.\numberthis
\end{align*}
\end{enumerate}
\end{proof}
\begin{cor}\label{cor:gamma is equivariant}
Given intertwiners $\rho_U:U\s U[1]$ and $\rho_V:V\s V[1]$, if we set
\begin{equation}\label{eq:rho_UV}
\rho_{U,V}:=\rho_U\odot\oone_V+\oone_U\odot\rho_V:U\otimes V\s(U\otimes V)[1]
\end{equation}
then we have the following facts:
\begin{enumerate}
\item[(i)] The intertwiner $\gamma_{U,V}$ in Lemma \ref{lem:gamma is natural} is an equivariant map as in \eqref{eq:equivariant as juxtaposition}.
\item[(ii)] The map $\pb{\rho}{\cdot}$ from Corollary \ref{cor:[rho,cdot] is derivation} is a derivation of the odot product \eqref{eq:odot of squiggly maps}, i.e. given also $\rho_{U'}:U'\s U'[1]$ and $\rho_{V'}:V'\s V'[1]$, 
\begin{equation}\label{eq:odot is cochain}
\pb{\rho}{f\odot g}=\pb{\rho}{f}\odot g+(-1)^{|f|}f\odot\pb{\rho}{g}\quad.
\end{equation}
\item[(iii)] If $\rho_U$ and $\rho_V$ are representations as in \eqref{eq:[rho,rho_V]} then so is \eqref{eq:rho_UV}.
\end{enumerate}
\end{cor}
\begin{proof}
\begin{enumerate}
\item[(i)] It is evident that
\begin{equation}\label{eq:varrho annihilates gamma}
\varrho_{\gamma_{U,V}}=0
\end{equation}
hence the claim reduces to $\gamma_{U,V}\rho_{U,V}=\rho_{V,U}\gamma_{U,V}$
but this follows from $\gamma$'s naturality \eqref{eq:gamma's naturality}.
\item[(ii)] We obviously have
\begin{subequations}\label{subeq:left unity of varrho and ell}
\begin{equation}\label{eq:varrho left unity}
\varrho_{f\odot\oone_V}=\varrho_f\odot\oone_V
\end{equation}
and 
\begin{equation}\label{eq:ell_UV}
\ell_{U\otimes V}=\ell_U\odot\oone_V
\end{equation}
\end{subequations}
thus
\begin{subequations}\label{subeq:[rho,fx1] and [rho,1xg]}
\begin{align*}
\pb{\rho}{f\odot\oone_V}\overset{\eqref{eq:diff of squiggly map}}{=}&\rho_{U',V}(f\odot\oone_V)-(-1)^{|f|}(f\odot\oone_V)\rho_{U,V}-\varrho_{f\odot\oone_V}\ell_{U\otimes V}\\
\overset{\eqref{subeq:left unity of varrho and ell}}{=}&(\rho_{U'}\odot\oone_V+\oone_{U'}\odot\rho_V)(f\odot\oone_V)-(-1)^{|f|}(f\odot\oone_V)(\rho_U\odot\oone_V+\oone_U\odot\rho_V)\\&-(\varrho_f\odot\oone_V)(\ell_U\odot\oone_V)\\
\overset{\eqref{eq:functoriality of odot}}{=}&\pb{\rho}{f}\odot\oone_V\numberthis
\end{align*}
but we can also use the fact that $\gamma$ is a natural isomorphism whose components are equivariant maps to immediately derive 
\begin{equation}
\pb{\rho}{\oone_U\odot g}=\oone_U\odot\pb{\rho}{g}
\end{equation}
\end{subequations}
hence
\begin{align*}
\pb{\rho}{f\odot g}\overset{\eqref{eq:functoriality of odot}}{=}&\pb{\rho}{(f\odot\oone_{V'})(\oone_U\odot g)}\\
\overset{\eqref{eq:[rho,gf]=...}}{=}&\pb{\rho}{f\odot\oone_{V'}}(\oone_U\odot g)+(-1)^{|f|}(f\odot\oone_{V'})\pb{\rho}{\oone_U\odot g}\\
\overset{\eqref{subeq:[rho,fx1] and [rho,1xg]}}{=}&(\pb{\rho}{f}\odot\oone_{V'})(\oone_U\odot g)+(-1)^{|f|}(f\odot\oone_{V'})(\oone_U\odot\pb{\rho}{g})\\
\overset{\eqref{eq:functoriality of odot}}{=}&\pb{\rho}{f}\odot g+(-1)^{|f|}f\odot\pb{\rho}{g}\numberthis\quad.
\end{align*}
\item[(iii)] We simply compute
\begin{align*}
\pb{\rho}{\rho_U\odot\oone_V+\oone_U\odot\rho_V}\overset{\eqref{eq:odot is cochain}}{=}&\pb{\rho}{\rho_U}\odot\oone_V+\oone_U\odot\pb{\rho}{\rho_V}\\
\overset{\eqref{eq:[rho,rho_V]}}{=}&\rho_U\rho_U\odot\oone_V+\oone_U\odot\rho_V\rho_V\\
\overset{\eqref{eq:functoriality of odot}}{=}&(\rho_U\odot\oone_V+\oone_U\odot\rho_V)(\rho_U\odot\oone_V+\oone_U\odot\rho_V)\numberthis\quad.
\end{align*}
\end{enumerate}
\end{proof}

\begin{propo}\label{propo: sym dg}
The dg-category $\Rep$ of Proposition \ref{propo:Rep is a dg-cat} is a symmetric strict monoidal dg-category as in Definition \ref{def:symmetric strict monoidal dg-category}:
\begin{enumerate}
\item[(i)] Given representations $U=(U,\rho_U)$ and $V=(V,\rho_V)$, their monoidal product $U\odot V$ is given by the tensor product $U\otimes V$ with the structure of representation given by \eqref{eq:rho_UV}.
\item[(ii)] The monoidal product of intertwiners is given by the odot product \eqref{eq:odot of squiggly maps}.
\item[(iii)] The monoidal unit is the ground field $\bbK$ equipped with the trivial representation $\rho_\bbK=0$.
\item[(iv)] The symmetric braiding $\gamma:\odot\Rightarrow\odot^\op$ is simply given by \eqref{eq:symmetric braiding of dgRep}.
\end{enumerate}
\end{propo}
\begin{proof}
We work backwards:
\begin{enumerate}
\item[(iv)] Lemma \ref{lem:gamma is natural} showed naturality and involutivity of the intertwiner $\gamma_{U,V}$ whereas item (i) of Corollary \ref{cor:gamma is equivariant} showed that $\gamma_{U,V}$ is an equivariant map with respect to the choice of $\rho_{U,V}$ in \eqref{eq:rho_UV}. In addition to this, it is obvious that the hexagon axiom is satisfied, i.e.
\begin{align}
\gamma_{U\odot V,W}=(\gamma_{U,W}\odot\oone_V)(\oone_U\odot \gamma_{V,W})\quad.
\end{align}
\item[(iii)] We evidently have $f\odot\oone_\bbK=f=\oone_\bbK\odot f$ thus $\rho_{U,\bbK}=\rho_U=\rho_{\bbK,U}$ hence $U\odot\bbK=U=\bbK\odot U$.
\item[(ii)] Item (i) of Corollary \ref{cor:odot is functorial} showed that the odot product satisfies the middle-four interchange law but we also clearly have preservation of units $\oone_U\odot\oone_V=\oone_{U\odot V}$. Item (ii) of Corollary \ref{cor:odot is functorial} showed that the odot product is associative thus giving $\rho_{U\odot V,W}=\rho_{U,V\odot W}$ hence $(U\odot V)\odot W=U\odot(V\odot W)$. Item (ii) of Corollary \ref{cor:gamma is equivariant} showed that the odot product is a cochain map between hom-complexes.
\item[(i)] Finally, item (iii) of Corollary \ref{cor:gamma is equivariant} shows that $\rho_{U,V}$ is a representation.
\end{enumerate}
\end{proof}

\subsection{Infinitesimal 2-braidings from 2-shifted Poisson structures}\label{subsec:Poisson inf2bra}
Throughout this subsection we use the adjoint representation from Example \ref{ex:adjoint representation}, denoted as $\g$. By $L_\g:\Rep\to\Rep$, we mean the following dg-functor:
\begin{subequations}
\begin{alignat}{6}
L_\g(U):=&\g\odot U\quad&&,\\
L_\g(f):=&\oone_\g\odot f&&.
\end{alignat}
\end{subequations}
We wish to define a dg-pseudonatural transformation $\varrho:L_\g\Rightarrow\id_{\Rep}:\Rep\to\Rep$ as in Definition \ref{def:dg-pseudonatural transformation}; in fact, Lemma \ref{lem:varrho map is a derivation} already demonstrated that \eqref{eq:derivation axiom of dg-pseudo} holds.
\begin{lem}\label{lem: action map is g linear}
For every representation $U$, we can define an equivariant map $\varrho_U:\g\odot U\s U$ as
\begin{align}\label{eq:cochain component of varrho psnat}
\left\{\varrho_U^i\overset{\eqref{eq:def of varrho^i_V}}{=}(-1)^{\i}\rho_U^{i+1}:\g^{\otimes\i}\otimes\g\otimes U\to U\big[-\i\big]\right\}_{i\geq1}
\end{align} 
and, given another representation $V$ and an intertwiner $f:U\s V\big[|f|\big]$, we have
\begin{align}\label{eq:varrho is pseudonatural}
f\varrho_U-\varrho_V(\oone_\g\odot f)=\pb{\rho}{\varrho_f}+\varrho_{\pb{\rho}{f}}\quad.
\end{align}
\end{lem}
\begin{proof}
We first demonstrate that the pseudonaturality relation \eqref{eq:varrho is pseudonatural} holds,
\begin{align*}
\pb{\rho}{\varrho_f}+\varrho_{\pb{\rho}{f}}\overset{\eqref{eq:diff of squiggly map}}{=}\rho_V\varrho_f&+(-1)^{|f|}\varrho_f\rho_{\g,U}-\varrho_{\varrho_f}\ell_{\g\otimes U}+\varrho_{\rho_Vf-(-1)^{|f|}f\rho_U-\varrho_f\ell_U}\numberthis\\
\overset{\eqref{eq:rho_UV},\eqref{eq:ell_UV},\eqref{eq:varrho is associative}}{=}&\rho_V\varrho_f+(-1)^{|f|}\varrho_f(\rho_\g\odot\oone_U+\oone_\g\odot\rho_U)-\varrho_{\varrho_f}(\ell_\g\odot\oone_U)+\varrho_{\rho_V}(\oone_\g\odot f)\\&-\rho_V\varrho_f-(-1)^{|f|}\varrho_f(\oone_\g\odot\rho_U)-f\varrho_{\rho_U}-\varrho_{\varrho_f}(\oone_\g\odot\ell_U)+(-1)^{|f|}\varrho_f\varrho_{\ell_U}
\end{align*}
however it is straightforward and trivial to check that
\begin{align}
\varrho_f(\rho_\g\odot\oone_U+\varrho_{\ell_U})=0=\varrho_{\varrho_f}(\ell_\g\odot\oone_U+\oone_\g\odot\ell_U)\quad,
\end{align}
leaving us with 
\begin{equation}\label{eq:proof varrho is pseudo}
\pb{\rho}{\varrho_f}+\varrho_{\pb{\rho}{f}}=\varrho_{\rho_V}(\oone_\g\odot f)-f\varrho_{\rho_U}
\end{equation}
but we clearly have
\begin{equation}\label{eq:varrho_U=-varrho_rho_U}
\varrho_U=-\varrho_{\rho_U}\quad.
\end{equation}
Now we show that the components \eqref{eq:cochain component of varrho psnat} constitute an equivariant map,
\begin{align}
\pb{\rho}{\varrho_U}\overset{\eqref{eq:varrho_U=-varrho_rho_U}}{=}-\pb{\rho}{\varrho_{\rho_U}}\overset{\eqref{eq:proof varrho is pseudo},\eqref{eq:[rho,rho_V]}}{=}\varrho_{\rho_U\rho_U}-\varrho_{\rho_U}(\oone_\g\odot\rho_U)+\rho_U\varrho_{\rho_U}
\overset{\eqref{eq:varrho is associative}}{=}0\,.
\end{align}
\end{proof}
\begin{lem}
Given intertwiners $f:U\s U'\big[|f|\big]$ and $g:V\s V'\big[|g|\big]$, we have 
\begin{align}
\varrho_{f\odot g}=&\,\varrho_f\odot g+(-1)^{|f|}(f\odot\varrho_g)(\gamma_{\g,U}\odot\oone_V)\nn\\
=&\,\varrho_f\odot g+(-1)^{|f||g|}\gamma_{V',U'}(\varrho_g\odot f)(\oone_\g\odot \gamma_{U,V})\quad.\label{eq:varrho splits monprods}
\end{align}
\end{lem}
\begin{proof}
First off, we note that
\begin{equation}\label{eq:varrho right unity}
\varrho_{\oone_U\odot g}\overset{\eqref{eq:gamma's naturality}}{=}\varrho_{\gamma_{V',U}(g\odot\oone_U)\gamma_{U,V}}\overset{\eqref{eq:varrho is associative}}{=}\gamma_{V',U}\varrho_{g\odot\oone_U}(\oone_\g\odot\gamma_{U,V})\overset{\eqref{eq:varrho left unity}}{=}\gamma_{V',U}(\varrho_g\odot\oone_U)(\oone_\g\odot\gamma_{U,V})
\end{equation}
hence
\begin{align*}
\varrho_{f\odot g}\overset{\eqref{eq:functoriality of odot}}{=}\quad~&\varrho_{(f\odot\oone_{V'})(\oone_U\odot g)}\\
\overset{\eqref{eq:varrho is associative}}{=}\quad~&\varrho_{f\odot\oone_{V'}}(\oone_\g\odot[\oone_U\odot g])+(-1)^{|f|}(f\odot\oone_{V'})\varrho_{\oone_U\odot g}\\
\overset{\eqref{eq:varrho left unity},\eqref{eq:varrho right unity}}{=}&(\varrho_f\odot\oone_{V'})(\oone_\g\odot[\oone_U\odot g])+(-1)^{|f|}(f\odot\oone_{V'})\gamma_{V',U}(\varrho_g\odot\oone_U)(\oone_\g\odot\gamma_{U,V})\\
\overset{\eqref{eq:odot is associative},\eqref{eq:gamma's naturality}}{=}\,&\varrho_f\odot g+(-1)^{|f|}(f\odot\varrho_g)(\gamma_{\g,U}\odot\oone_V)
\end{align*}
and the second equality of \eqref{eq:varrho splits monprods} follows by using the fact that $\gamma$ is a symmetric braiding.
\end{proof}
\begin{cor}\label{cor:varrho monoidal properties}
Given representations $U$ and $V$, we have
\begin{align}\label{eq:varrho_UV}
\varrho_{U\odot V}=\varrho_U\odot\oone_V+\gamma_{V,U}(\varrho_V\odot\oone_U)(\oone_\g\odot \gamma_{U,V})\quad.
\end{align}
\end{cor}
\begin{proof}
We compute
\begin{align}
\varrho_{U\odot V}\overset{\eqref{eq:varrho_U=-varrho_rho_U},\eqref{eq:rho_UV}}{=}-\varrho_{\rho_U\odot\oone_V+\oone_U\odot\rho_V}\overset{\eqref{eq:varrho splits monprods},\eqref{eq:varrho_U=-varrho_rho_U}}{=}\varrho_U\odot\oone_V+\gamma_{V,U}(\varrho_V\odot\oone_U)(\oone_\g\odot \gamma_{U,V})\,.
\end{align}
\end{proof}
We can likewise define a dg-functor $R_\g:\Rep\to\Rep$ of right-multiplication by $\g$ and define a dg-pseudonatural transformation $\lambda:R_\g\Rightarrow\id_{\Rep}:\Rep\to\Rep$ as follows: for every representation $U$, we define an equivariant map 
\begin{subequations}\label{subeq:lambda data}
\begin{equation}\label{eq:lambda_U:=}
\lambda_U:=\varrho_U\gamma_{U,\g}:U\odot\g\s U
\end{equation} 
and for every intertwiner $f:U\s U'\big[|f|\big]$, we set
\begin{equation}\label{eq:lambda_f:=}
\lambda_f:=\varrho_f\gamma_{U,\g}:U\odot\g\s U'\big[|f|-1\big]\quad.
\end{equation}
\end{subequations}
It is a simple matter to show that this data satisfies analogous relations to $\varrho$ but we will only make use of the following two:
\begin{alignat}{6}
\lambda_{gf}=&\,\lambda_g(f\odot\oone_\g)+(-1)^{|g|}g\lambda_f\quad&&,\label{eq:lambda is associative}\\
\lambda_{f\odot g}=&\,(\lambda_f\odot g)(\oone_U\odot \gamma_{V,\g})+(-1)^{|f|}f\odot\lambda_g&&.\label{eq:lambda_f,g}
\end{alignat}
We will also use the following lemma together with its corollary in the proof that the infinitesimal 2-braiding is coherent.
\begin{lem}
For every representation $U$,
\begin{equation}\label{eq:varrho is skew-symmetric}
\varrho_{\varrho_U}+\varrho_{\varrho_U}(\gamma_{\g,\g}\odot\oone_U)=0\quad.
\end{equation}
\end{lem}
\begin{proof}
This claim is merely a matter of recalling that the structure maps $\rho_U^i$ are totally skew-symmetric as in Definition \ref{def:squiggly notation}.
\end{proof}
\begin{cor}
For every representation $U$,
\begin{subequations}\label{subeq:skew-sym of varrho and lambda}
\begin{alignat}{6}
\varrho_{\lambda_U}+\lambda_{\varrho_U}=&\,0\quad&&,\label{eq:varrho_lambda_U+lambda_varrho_U=0}\\
\lambda_{\lambda_U}+\lambda_{\lambda_U}(\oone_U\odot \gamma_{\g,\g})=&\,0&&.\label{eq:lambda is skew-symmetric}
\end{alignat}
\end{subequations}
\end{cor}
\begin{proof}
Simply use the definition of $\lambda$ in \eqref{subeq:lambda data} together with the fact that $\varrho$ is a derivation \eqref{eq:varrho is associative} and annihilates $\gamma$ as in \eqref{eq:varrho annihilates gamma}. 
\end{proof}
Now we put to work the formalism developed thus far by bringing the 2-shifted Poisson structure from Remark \ref{rem:2SPS} into the fray. 
\begin{lem}\label{lem:pi_2 is an intertwiner}
Given a 2-shifted Poisson structure \eqref{subeq:2SPS}, consider the weight 2 term \eqref{eq:weight 2 of 2SPS} and set 
\begin{equation}\label{eq:ell_2^i:=}
\left\{\varpi_2^i:=(-1)^i\pi_2^{\i}:{\Mwedge}^{\i}\g\to(\Sym^2\g)[-\i]\right\}_{i\in\bbN}\quad,
\end{equation}
this defines an equivariant map $\varpi_2:\bbK\s\Sym^2\g$. Consider the weight 3 term \eqref{eq:weight 3 of 2SPS} and set
\begin{equation}\label{eq:ell_3^i:=}
\left\{\varpi_3^i:=(-1)^i\pi_3^{\i}:{\Mwedge}^{\i}\g\to(\Sym^3\g)[-2-\i]\right\}_{i\in\bbN}\quad,
\end{equation}
this defines an intertwiner $\varpi_3:\bbK\s(\Sym^3\g)[-2]$ such that
\begin{equation}\label{eq:ell_3 is 2-homotopy}
\pb{\rho}{\varpi_3}=\Sigma_{2,1}^+(\varrho_{\varpi_2}\odot\oone_\g)\varpi_2\quad,
\end{equation}
where $\Sigma_{2,1}^+:=\{\Sigma_{2,1}^+:(\Sym^2\g)\otimes\g\to\Sym^3\g\}:(\Sym^2\g)\odot\g\s\Sym^3\g$ is an equivariant map. 
\end{lem}
\begin{proof}
First we prove that \eqref{eq:ell_2^i:=} defines an equivariant map as in Definition \ref{defi: weak rep}; we have
\begin{align}
\pb{\rho}{\varpi_2}^{i+1}&\overset{\eqref{eq:arity component of [rho,f]}}{=}\sum_{\j+k+1=i+1}(-1)^{jk}\rho_{\g,\g}^j[1_{\j}\otimes\varpi_2^{k+1}]\Sigma_{\j,k}-\sum_{j+k=i+1}(-1)^{\j\k}\varpi_2^{j+1}(1_{\j}\otimes\pi_1^k)\Sigma_{\j,k}\nn\\
&\overset{\eqref{eq:ell_2^i:=}}{=}(-1)^i\sum_{\j+k=i}(-1)^{j\k}\left([\pi_1^j\otimes1_\g+(1_\g\otimes\pi_1^j)(s_{\j,\g}\otimes1_\g)][1_{\j}\otimes\pi_2^k]+\pi_2^j[1_{\j}\otimes\pi_1^k]\right)\Sigma_{\j,k}\nn\\
&\overset{\eqref{eq:weight 2 MC}}{=}0\quad.
\end{align}
Let us quickly remark that it is evident, using \eqref{eq:arity component of [rho,f]}, that $\Sigma^+_{2,1}$ is actually an equivariant map; now we prove that relation \eqref{eq:ell_3 is 2-homotopy} holds,
\begin{align*}
[\Sigma_{2,1}^+(\varrho_{\varpi_2}\odot\oone_\g)\varpi_2]^i\overset{\eqref{eq:homotoper component of varrho psnat}}{=}&\sum_{\j+\k=i}(-1)^{\j k}\Sigma_{2,1}^+[\varpi_2^j\otimes1_\g][1_{\j-1}\otimes\varpi_2^k]\Sigma_{\j-1,\k}\\
\overset{\eqref{eq:weight 3 of 2SPS}}{=}&(-1)^{\i}\sum_{\j+k=i}(-1)^{\j k}(\Sigma_{1,2}^+[\pi_1^j\otimes1_{\g^{\otimes2}}][1_{\j}\otimes\varpi_3^k]\Sigma_{\j,\k}-\varpi_3^j[1_{\j-1}\otimes\pi_1^k]\Sigma_{\j-1,\k})\\
\overset{\eqref{eq:arity component of [rho,f]}}{=}&[\rho,\varpi_3]^i\quad.\numberthis
\end{align*}
\end{proof}
\begin{cor}
Given representations $U$ and $V$, if we set
\begin{equation}
\varpi^2_{U,V}:=\oone_U\odot\varpi_2\odot\oone_V:U\odot V\s U\odot\g\odot\g\odot V
\end{equation}
then this defines a dg-natural transformation $\varpi^2:\odot\Rightarrow\odot(R_\g\cmp L_\g)$.
\end{cor}
\begin{constr}\label{con:defining t}
Denoting $\CC:=\Ho_2[\Rep]$, we define $t:\odot\Rightarrow\odot:\CC\smallbox\CC\to\CC$ as the following vertical composite,
\begin{equation}
\begin{tikzcd}
	\odot && \odot \\
	& {\odot(R_\g\smallbox L_\g)}
	\arrow["t", Rightarrow, dashed, from=1-1, to=1-3]
	\arrow["{\varpi^2}"', Rightarrow, from=1-1, to=2-2]
	\arrow["{\Id_\odot*(\lambda\smallbox\varrho)}"', Rightarrow, from=2-2, to=1-3]
\end{tikzcd}\qquad,
\end{equation}
i.e. we define
\begin{equation}\label{eq:t as vertical composite}
t:=(\Id_\odot*[\lambda\smallbox\varrho])\circ\varpi^2\quad.
\end{equation}
Given representations $U$ and $V$, we use \eqref{eq:object components of vercomp}, \eqref{eq:object components of horcomp} and \eqref{eq:object components of moncomp} to derive
\begin{equation}\label{eq:t_U,V}
t_{U,V}:=(\lambda_U\odot\varrho_V)(\oone_U\odot\varpi_2\odot\oone_V):U\odot V\s U\odot V\quad.
\end{equation}
Given equivariant maps $f:U\s U'$ and $g:V\s V'$, we use \eqref{eq:hom components of vercomp pseudos}, \eqref{eq:homotoper components of horcomp pseudos} and \eqref{eq:homotoper components of moncomp pseudos} to derive
\begin{equation}\label{eq:t_(f,g)}
t_{f,g}:=\left(\lambda_f\odot g\varrho_V+\lambda_{U'}[f\odot\oone_\g]\odot\varrho_g\right)(\oone_U\odot\varpi_2\odot\oone_V):U\odot V\s U'\odot V'[-1]\quad.
\end{equation}
\end{constr}
\begin{theo}\label{theo:2SPS induces coherent inf2bra}
The pseudonatural transformation $t:\odot\Rightarrow\odot$ of Construction \ref{con:defining t} defines a totally $\gamma$-equivariant coherent infinitesimal 2-braiding as in Definition \ref{def: coherency}.
\end{theo}
\begin{proof}
Let us first show $\gamma$-equivariance as in Definition \ref{def:sym t}; we note that
\begin{subequations}\label{subeq:proof of t_U,V symmetry}
\begin{equation}
\gamma_{U,V}(\lambda_U\odot\varrho_V)\overset{\eqref{eq:lambda_U:=}}{=}\gamma_{U,V}(\varrho_U\odot\lambda_V)(\gamma_{U,\g}\odot \gamma_{\g,V})\overset{\eqref{eq:gamma's naturality}}{=}(\lambda_V\odot\varrho_U)\gamma_{\g U,V\g}(\gamma_{U,\g}\odot \gamma_{\g,V})
\end{equation}
and
\begin{equation}
(\oone_U\odot\varpi_2\odot\oone_V)\gamma_{V,U}=\gamma_{\g\g V,U}(\varpi_2\odot\oone_{UV})=\gamma_{\g\g V,U}(\gamma_{V,\g\g}\odot\oone_U)(\oone_V\odot\varpi_2\odot\oone_U)
\end{equation}
\end{subequations}
thus
\begin{equation}
\gamma_{U,V}t_{U,V}\gamma_{V,U}\overset{\eqref{subeq:proof of t_U,V symmetry}}{=}(\lambda_V\odot\varrho_U)(\oone_V\odot \gamma_{\g,\g}\odot\oone_U)(\oone_V\odot\varpi_2\odot\oone_U)=t_{V,U}\quad.
\end{equation}
To prove the second relation \eqref{intertwine homotopy}, we simply note
\begin{equation}
\big[\Id_\odot*(\lambda\smallbox\varrho)\big]_{f,g}\overset{\eqref{eq:alternative moncomp of pseudos}}{=}f\lambda_U\odot\varrho_g+\lambda_f\odot\varrho_{V'}[\oone_\g\odot g]=(f\varrho_U\odot\lambda_g+\varrho_f\odot\lambda_{V'}[g\odot\oone_\g])(\gamma_{U,\g}\odot \gamma_{\g,V})\,.
\end{equation}
Let us now show that the left infinitesimal hexagon relations \eqref{subeq:t_U(VW)} are satisfied; we have 
\begin{align*}
t_{U,V\odot W}\overset{\eqref{eq:varrho_UV}}{=}&(\lambda_U\odot\varrho_V\odot\oone_W+\lambda_U\odot \gamma_{W,V}[\varrho_W\odot\oone_V][\oone_\g\odot \gamma_{V,W}])(\oone_U\odot\varpi_2\odot\oone_{V\odot W})\\
\overset{\eqref{eq:naturality of symmetric braiding}}{=}&[\lambda_U\odot\varrho_V][\oone_U\odot\varpi_2\odot\oone_V]\odot\oone_W\\&+[\oone_U\odot \gamma_{W,V}][(\lambda_U\odot\varrho_W)(\oone_U\odot\varpi_2\odot\oone_W)\odot\oone_V][\oone_U\odot \gamma_{V,W}]\numberthis
\end{align*}
and
\begin{align*}
t_{f,g\odot h}\overset{\eqref{eq:varrho splits monprods}}{=}&\Big(\lambda_f\odot(g\odot h)(\varrho_V\odot\oone_W+\gamma_{W,V}[\varrho_W\odot\oone_V][\oone_\g\odot \gamma_{V,W}])\numberthis\\&+\lambda_{U'}(f\odot\oone_\g\big)\odot(\varrho_g\odot h+\gamma_{W',V'}[\varrho_h\odot g][\oone_\g\odot \gamma_{V,W}])\Big)\big(\oone_U\odot\varpi_2\odot\oone_{VW}\big)\\
\overset{\eqref{eq:naturality of symmetric braiding}}{=}&(\lambda_f\odot g\varrho_V+\lambda_{U'}[f\odot\oone_\g]\odot\varrho_g)(\oone_U\odot\varpi_2\odot\oone_V)\odot h\\
&+\big(\oone_{U'}\odot \gamma_{W',V'}\big)\Big((\lambda_f\odot h\varrho_W+\lambda_{U'}[f\odot\oone_\g]\odot\varrho_h)(\oone_U\odot\varpi_2\odot\oone_W)\odot g\Big)\big(\oone_U\odot \gamma_{V,W}\big)\,.
\end{align*}
Using \eqref{eq:varrho kills units} and \eqref{eq:varrho annihilates gamma}, we have $t_{\gamma_{U,V},\oone_W}=0$ hence the infinitesimal 2-braiding is indeed totally $\gamma$-equivariant. We now prove that $t$ is coherent, we have
\begin{equation}
t_{t_{U,V},\oone_W}\overset{\eqref{eq:varrho kills units}}{=}(\lambda_{t_{U,V}}\odot\varrho_W)(\oone_{UV}\odot\varpi_2\odot\oone_W)
\end{equation}
and we can expand $\lambda_{t_{U,V}}$ as follows
\begin{align}
\lambda_{t_{U,V}}\overset{\eqref{eq:lambda is associative}}{=}&\lambda_{\lambda_U\odot\varrho_V}(\oone_U\odot\varpi_2\odot\oone_{V\g})+(\lambda_U\odot\varrho_V)\lambda_{\oone_U\odot\varpi_2\odot\oone_V}\\
\overset{\eqref{eq:lambda_f,g}}{=}&([\lambda_{\lambda_U}\odot\varrho_V][\oone_{U\g}\odot \gamma_{\g V,\g}]+\lambda_U\odot\lambda_{\varrho_V})(\oone_U\odot\varpi_2\odot\oone_{V\g})\nn\\&+(\lambda_U\odot\varrho_V)(\oone_U\odot[\lambda_{\varpi_2}\odot\oone_V]\gamma_{V,\g})\nn
\end{align}
but we obviously have $\lambda_{\varpi_2}=\varrho_{\varpi_2}$ so
\begin{align}
t_{t_{U,V},\oone_W}=\,&(\lambda_U\odot\lambda_{\varrho_V}\odot\varrho_W+[\lambda_{\lambda_U}\odot\varrho_V][\oone_{U\g}\odot \gamma_{\g V,\g}]\odot\varrho_W)(\oone_U\odot\varpi_2\odot\oone_V\odot\varpi_2\odot\oone_W)\nn\\&+(\lambda_U\odot\varrho_V\odot\varrho_W)(\oone_U\odot \gamma_{V,\g\g}\odot\oone_{\g W})(\oone_{UV}\odot[\varrho_{\varpi_2}\odot\oone_\g]\varpi_2\odot\oone_W)\quad.
\end{align}
Likewise,
\begin{equation}
t_{\oone_U,t_{V,W}}\overset{\eqref{eq:varrho kills units}}{=}(\lambda_U\odot\varrho_{t_{V,W}})(\oone_{U}\odot\varpi_2\odot\oone_{VW})
\end{equation}
and
\begin{align}
\varrho_{t_{V,W}}\overset{\eqref{eq:varrho is associative}}{=}&\varrho_{\lambda_V\odot\varrho_W}(\oone_{\g V}\odot\varpi_2\odot\oone_W)+(\lambda_V\odot\varrho_W)\varrho_{\oone_V\odot\varpi_2\odot\oone_W}\\
\overset{\eqref{eq:varrho splits monprods}}{=}&(\varrho_{\lambda_V}\odot\varrho_W+[\lambda_V\odot\varrho_{\varrho_W}][\gamma_{\g,V\g}\odot\oone_{\g W}])(\oone_{\g V}\odot\varpi_2\odot\oone_W)\nn\\&+(\lambda_V\odot\varrho_W)([\oone_V\odot\varrho_{\varpi_2}]\gamma_{\g,V}\odot\oone_W)\nn
\end{align}
so 
\begin{align}\label{eq:t_23}
t_{\oone_U,t_{V,W}}=\,&(\lambda_U\odot\varrho_{\lambda_V}\odot\varrho_W+\lambda_U\odot[\lambda_V\odot\varrho_{\varrho_W}][\gamma_{\g,V\g}\odot\oone_{\g W}])(\oone_U\odot\varpi_2\odot\oone_V\odot\varpi_2\odot\oone_W)\nn\\&+(\lambda_U\odot\lambda_V\odot\varrho_W)(\oone_{U\g}\odot \gamma_{\g\g,V}\odot\oone_W)(\oone_U\odot[\oone_\g\odot\varrho_{\varpi_2}]\varpi_2\odot\oone_{VW})\quad.
\end{align}
With a view to later using \eqref{eq:ell_3 is 2-homotopy}, we rewrite the last term of \eqref{eq:t_23} as
\begin{align}
&(\lambda_U\odot\varrho_V\odot\varrho_W)(\oone_{U\g}\odot [\gamma_{V,\g}\odot\oone_\g]\gamma_{\g\g,V}\odot\oone_W)(\oone_U\odot \gamma_{V,\g\g\g}\odot\oone_W)(\oone_{UV}\odot[\oone_\g\odot\varrho_{\varpi_2}]\varpi_2\odot\oone_W)\nn\\&=(\lambda_U\odot\varrho_V\odot\varrho_W)(\oone_U\odot \gamma_{V,\g\g}\odot\oone_{\g W})(\oone_{UV}\odot[\oone_\g\odot\varrho_{\varpi_2}]\varpi_2\odot\oone_W)\quad.
\end{align}
A straightforward (but tedious) exercise in using the coherency of $\gamma$ reveals that we have the following expression for $(\oone_U\odot \gamma_{W,V})t_{t_{U,W},\oone_V}(\oone_U\odot \gamma_{V,W})$,
\begin{align}
&(\lambda_U\odot[\lambda_V\odot\varrho_{\varrho_W}][\gamma_{\g,V\g\g}\odot\oone_W]+[\lambda_{\lambda_U}\odot\varrho_V][\oone_U\odot \gamma_{\g\g V,\g}]\odot\varrho_W)(\oone_U\odot\varpi_2\odot\oone_V\odot\varpi_2\odot\oone_W)\nn\\&\qquad+(\lambda_U\odot\varrho_V\odot\varrho_W)(\oone_U\odot \gamma_{V,\g\g}\odot\oone_{\g W})(\oone_{UV}\odot[\oone_\g\odot \gamma_{\g,\g}][\varrho_{\varpi_2}\odot\oone_\g]\varpi_2\odot\oone_W)\quad.
\end{align}
Finally, using \eqref{eq:ell_3 is 2-homotopy} together with \eqref{eq:varrho is skew-symmetric} and \eqref{subeq:skew-sym of varrho and lambda} gives us
\begin{align}
t_{t_{U,V},\oone_W}+t_{\oone_U,t_{V,W}}+(\oone_U\odot \gamma_{W,V})t_{t_{U,W},\oone_V}(\oone_U\odot \gamma_{V,W})=\pb{\rho}{t_{U,V,W}}\quad,
\end{align}
where 
\begin{equation}
t_{U,V,W}:=([\lambda_U\odot\varrho_V][\oone_U\odot \gamma_{V,\g\g}]\odot\varrho_W)(\oone_{UV}\odot\varpi_3\odot\oone_W)
\end{equation}
is a degree 2-homotopy $t_{U,V,W}:U\odot V\odot W\s U\odot V\odot W[-2]$ thus killed by the truncation.
\end{proof}

\newpage
\section{Representation theory as derived geometry}\label{sec:geometric model}
Throughout this section, we work with only finite-dimensional homotopy Lie algebras $\g=(\g,\ell)$. The objective of this section is to show that the symmetric monoidal dg-category $\Rep$ of Proposition \ref{propo: sym dg} admits a symmetric monoidal dg-equivalence (as in Definition \ref{def:sym mon dg-equivalence}) to a symmetric monoidal dg-category of certain dg-modules. Subsection \ref{subsec:CDGAs} begins by recalling the definition of finitely generated semi-free CDGAs $A$ from \cite[Definition 2.1]{KLS25}. Remark \ref{rem:need completed semi-free CDGAs} cautions that such CDGAs are not `big enough' to correspond to finite-dimensional homotopy Lie algebras $\g$ hence Definition \ref{def:total completed semi-free CDGA} introduces the notion of a completed semi-free CDGA. Proposition \ref{propo:CE of fd L infinity algebra} shows that finite-dimensional homotopy Lie algebras are in one-to-one correspondence with augmented finitely generated completed semi-free CDGAs. 
\sk

Subsection \ref{subsec:equivalence} begins by recalling from \cite[Subsection 2.2]{KLS24} the symmetric monoidal dg-category ${}_A\dgMod^\sf$ of semi-free dg-modules over a CDGA $A$. Likewise to the above-mentioned, Remark \ref{rem:curved representation -> no mon prod} cautions that semi-free dg-modules are similarly not `big enough' to correspond to representations hence Definition \ref{def:completed semi-free dg-module} introduces the notion of a completed semi-free dg-module over a completed semi-free CDGA. Proposition \ref{propo:CE of rep} shows that representations are in one-to-one correspondence with completed semi-free dg-modules over the Chevalley-Eilenberg algebra. Finally, Proposition \ref{propo:constructed sym mon dg-equivalence} explicitly constructs the symmetric monoidal dg-equivalence $\CE_\g:\Rep\simeq{}_{\CE_\g}\dgMod^\sf$.

\subsection{Chevalley-Eilenberg algebras}\label{subsec:CDGAs}
A \textbf{CDGA} is a commutative monoid object in the symmetric monoidal category $\Ch$. More explicitly, a CDGA is a triple $A= (A,\mu,\eta)$ consisting of a cochain complex $A=(A^\sharp,\delta)\in\Ch$
and two cochain maps: 
\begin{subequations}
\begin{align}
\mu:A\otimes A\to A\qquad,\qquad a\otimes a'\mapsto a\,a'\,,
\end{align} 
(called \textbf{multiplication}) and
\begin{align}
\eta:\bbK\to A\qquad,\qquad
1\mapsto\eta\,,
\end{align}
\end{subequations}
(called \textbf{unit}) 
which satisfy the usual axioms of associativity,
\begin{subequations}
\begin{align}\label{eq:CDGA associativity}
a(a'a'')=(a\,a')a''\quad,
\end{align}
unitality,
\begin{align}\label{eq:CDGA unitality}
\eta a=a=a\eta\quad,
\end{align}
and commutativity,
\begin{align}\label{eq:CDGA commutativity}
a\,a'=(-1)^{|a||a'|}a'a\quad.
\end{align}
\end{subequations}
The multiplication being a cochain map means that the \textbf{Leibniz law} is satisfied,
\begin{subequations}
\begin{align}\label{eq:Leibniz single}
\delta(a\,a')=\delta(a)a'+(-1)^{|a|}a\delta(a'),
\end{align}
whereas the unit being a cochain map means that it is annihilated by the differential,
\begin{equation}
\delta(\eta)=0\quad.
\end{equation}
\end{subequations}
A CDGA $A$ is said to be \textbf{augmented} if there exists a DGA homomorphism $\varepsilon:A\to\bbK$, i.e. an algebra homomorphism $\varepsilon_{|A^0}:A^0\to\bbK$ which annihilates the 0-cocycles $B^0(A)$.  
\begin{defi}\label{def:semifreeCDGA}
A CDGA $A=(A,\mu,\eta)$ is called \textbf{semi-free}
if its underlying graded commutative algebra $A^\sharp=(A^\sharp,\mu,\eta)$ is free, 
i.e., there exists some graded vector space $U$ such that the following isomorphism holds between graded-commutative algebras,
\begin{align}
A^\sharp\cong\Sym(U)\quad.
\end{align}
A semi-free CDGA $A$ is called \textbf{finitely generated} if the space of generators $U$ is finite-dimensional. 
\end{defi}
\begin{rem}\label{rem:diff determined on sfCDGA}
The Leibniz law \eqref{eq:Leibniz single} implies that the differential $\delta$ of a semi-free CDGA $A^\sharp\cong\Sym(U)$ is completely determined by its restriction $\delta(u)$ to the generators. Hence, one can prescribe such differential by its values $\delta(u)$ on generators such that $|\delta(u)|=|u|+1$ and $\delta^2(u)=0$.
Thus, a generic semi-free CDGA is augmented by the canonical projection which maps generators to 0 if and only if, for every degree $-1$ generator $u$, the differential $\delta(u)$ has no scalar component.
\end{rem}
\begin{rem}\label{rem:need completed semi-free CDGAs}
Given a finite-dimensional homotopy Lie algebra $\g=(\g,\ell)$, we denote the basis of $\g$ as $\{x_\alpha\}_{1\leq\alpha\leq\dim(\g)}$ and the \textbf{structure constants} of $\ell^i:\bigwedge^i\g\to\g[2-i]$ as
\begin{align}
\ell^i(x_{\alpha_1},\ldots,x_{\alpha_i})=\ell^\alpha_{\alpha_1,\ldots,\alpha_i}x_\alpha\quad,
\end{align}
where the Einstein summation convention has support only over $1\leq\alpha\leq\dim(\g)$ such that
\begin{align}\label{eq:compact support of fd L alg maps}
|x_\alpha|=2-i+\sum_{l=1}^i|x_{\alpha_l}|\quad.
\end{align}
Denoting the dual basis as $\{x^\alpha\}_{1\leq\alpha\leq\dim(\g)}$ and the generators of $\Sym(\g^*[-1])$ as
\begin{align}\label{eq:theta^i=sx^i}
\theta^\alpha:=\varsigma x^\alpha:\g\to\bbK\big[1-|x_\alpha|\big]\quad,
\end{align}
we would like to construct the Chevalley-Eilenberg algebra as a finitely generated semi-free CDGA with underlying graded-commutative algebra $\Sym(\g^*[-1])$ and the differential given by dualising the structure maps $\big\{\ell^i:{\bigwedge}^i\g\to\g[2-i]\big\}_{i\geq1}$. The problem is that \eqref{eq:compact support of fd L alg maps} implies
\begin{align}
|\theta^{\alpha_1}\cdots\theta^{\alpha_i}|=|\theta^\alpha|+1
\end{align} 
hence the differential $\delta$ on generators $\theta^\alpha$ will be given, schematically, by
\begin{align}\label{eq:schematic CE differential}
\delta(\theta^\alpha)\,\sim\,\sum_{i=1}^\infty\ell^\alpha_{\alpha_1,\ldots,\alpha_i}\theta^{\alpha_1}\cdots\theta^{\alpha_i}\quad.
\end{align}
\end{rem}
The remedy is the following definition. 
\begin{defi}\label{def:total completed semi-free CDGA}
We say a CDGA $A$ is a \textbf{completed semi-free CDGA} if there exists a graded vector space $U$ and an isomorphism $A^\sharp\cong\widehat{\Sym}(U):=\prod_{i\in\bbN}\Sym^i(U)$ between graded-commutative algebras.
\end{defi}
We recall the generalised Jacobi identity \eqref{eq:generalised Jacobi identity} as
\begin{subequations}
\begin{align}\label{eq:curved Jacobi identity}
\forall k\geq1\,,\qquad\qquad\sum_{\i+j=k}(-1)^{\i j}\ell^i(1_{\i}\otimes\ell^j)\Sigma_{\i,j}=0\quad,
\end{align}
which then imposes the following relation between structure constants: $\forall k\geq1$ and for all $1\leq\alpha_1,\ldots,\alpha_k\leq\dim(\g)$,
\begin{align}\label{eq:indexed Jacobi}
\sum_{\i+j=k}\sum_{\sigma\in\S_{\i,j}}\varepsilon_\sigma(-1)^{j\big(\i+\big|x_{\alpha_{\sigma(1)}}\cdots x_{\alpha_{\sigma(\i)}}\big|\big)}\ell^\alpha_{\alpha_{\sigma(1)},\ldots,\alpha_{\sigma(\i)},\beta}\ell^\beta_{\alpha_{\sigma(i)},\ldots,\alpha_{\sigma(k)}}=0\quad.
\end{align}
\end{subequations}
\begin{constr}
The weight 1 Maurer-Cartan element $\pi_1\in\Pol(\CE_\g,n)^{n+2}$ deforms the zero differential of $\widehat{\Sym}(\g^*[-1])$ as 
\begin{align}
\delta\theta^\alpha=\pb{\pi_1}{\theta^\alpha}=(-1)^{|x_\alpha|+n}\theta^\alpha\bullet\pi_1=\sum_{i=1}^\infty(-1)^{|x_\alpha|+n+(1-|x_\alpha|)\i}\theta^\alpha\pi^i_1=\sum_{i=1}^\infty(-1)^{n+|x_\alpha|i}\theta^\alpha\ell^i\,.
\end{align}
\end{constr}
\begin{propo}\label{propo:CE of fd L infinity algebra}
Given a finite-dimensional graded vector space $\g$ together with a family of graded linear maps $\ell:=\left\{\ell^i:\bigwedge^i\g\to\g[2-i]\right\}_{i\geq1}$, define the finitely generated completed free CGA $\CE_\g:=\widehat{\Sym}(\g^*[-1])$ and the graded linear map $\delta:\CE_\g\to\CE_\g[1]$ on generators $\theta^\alpha$ as
\begin{align}\label{eq:CE differential on generator}
\delta\theta^\alpha:=\sum_{i=1}^\infty\frac{1}{i!}(-1)^{\sum_{l=1}^i|x_{\alpha_l}|(l+|x_{\alpha_{l+1}}\cdots x_{\alpha_i}|)}\ell^\alpha_{\alpha_1,\ldots,\alpha_i}\theta^{\alpha_1}\cdots\theta^{\alpha_i}\quad,
\end{align}
then $\delta$ is a differential structure if and only if $\ell$ satisfies the generalised Jacobi identity \eqref{eq:indexed Jacobi}.
\end{propo}
\begin{proof}
First we show well-definedness, the \textbf{reversal degree}
\begin{align}
|x_{\alpha_1}\cdots x_{\alpha_i}|_\rev:=\sum_{l=1}^i|x_{\alpha_l}||x_{\alpha_{l+1}}\cdots x_{\alpha_i}|
\end{align}
is symmetric $|x_{\alpha_1}\cdots x_{\alpha_l}x_{\alpha_{l+1}}\cdots x_{\alpha_i}|_\rev=|x_{\alpha_1}\cdots x_{\alpha_{l+1}}x_{\alpha_l}\cdots x_{\alpha_i}|_\rev$ and
\begin{align}
\ell^\alpha_{\alpha_1,\ldots,\alpha_i}\theta^{\alpha_1}\cdots\theta^{\alpha_i}=\textcolor{blue}{-(-1)^{|x_{\alpha_l}||x_{\alpha_{l+1}}|}}\ell^\alpha_{\alpha_1,\ldots,\alpha_{l+1},\alpha_l,\ldots,\alpha_i}\textcolor{blue}{(-1)^{|\theta^{\alpha_l}||\theta^{\alpha_{l+1}}|}}\theta^{\alpha_1}\cdots\theta^{\alpha_{l+1}}\theta^{\alpha_l}\cdots\theta^{\alpha_i}
\end{align}
but these \textcolor{blue}{signs} clearly give $(-1)^{|x_{\alpha_l}|+|x_{\alpha_{l+1}}|}$ hence \eqref{eq:CE differential on generator} is indeed an element of $\CE^\sharp(\g)$. 

Now let us show that the square-zero condition is upheld by expressing\footnote{Note that, in the notation below, the check $\check{\theta}^{\alpha_m}$ denotes omission of that term.} $\delta^2(\theta^\alpha)$, 
\begin{align}\label{eq:delta^2 theta^i}
&\sum_{i=1}^\infty\frac{1}{i!}(-1)^{|x_{\alpha_1}\cdots x_{\alpha_i}|_\rev+\sum_{l=1}^i|x_{\alpha_l}|l}\ell^\alpha_{\alpha_1,\ldots,\alpha_i}\delta(\theta^{\alpha_1}\cdots\theta^{\alpha_i})\\
&=\sum_{i=1}^\infty\sum_{m=1}^i\frac{(-1)^{|x_{\alpha_1}\cdot\cdot x_{\alpha_i}|_\rev+\sum_{l=1}^i|x_{\alpha_l}|l+|\theta^{\alpha_1}\cdot\cdot\theta^{\alpha_{m-1}}|+|\delta\theta^{\alpha_m}||\theta^{\alpha_{m+1}}\cdot\cdot\theta^{\alpha_i}|}}{i!}\ell^\alpha_{\alpha_1,..,\alpha_i}\theta^{\alpha_1}\cdot\cdot\check{\theta}^{\alpha_m}\cdot\cdot\theta^{\alpha_i}\delta(\theta^{\alpha_m})\nn\\
&=\sum_{i=1}^\infty\sum_{m=1}^i\frac{(-1)^{|x_{\alpha_1}\cdot\cdot x_{\alpha_i}|_\rev+\textcolor{violet}{\sum_{l=1}^i|x_{\alpha_l}|l+|\theta^{\alpha_1}\cdot\cdot\theta^{\alpha_{\m}}|+\widetilde{|x_{\alpha_m}|}(i-m)}}}{i!}\ell^\alpha_{\alpha_1,..,\check{\alpha}_m,..,\alpha_i,\alpha_m}\theta^{\alpha_1}\cdot\cdot\check{\theta}^{\alpha_m}\cdot\cdot\theta^{\alpha_i}\delta(\theta^{\alpha_m})\nn.
\end{align}
For $m=i$, the \textcolor{violet}{sign} is given by
\begin{align}\label{eq:m=i sign in CE alg proof}
(-1)^{\sum_{l=1}^i|x_{\alpha_l}|l+|\theta^{\alpha_1}\cdot\cdot\theta^{\alpha_{\i}}|}\quad.
\end{align}
For $1\leq m\leq i-1$, we recover \eqref{eq:m=i sign in CE alg proof} by reindexing cyclically as:
\begin{align}
\alpha_m\mapsto\alpha_i\quad,\quad\alpha_{m+1}\mapsto\alpha_m\quad,\quad\alpha_{m+2}\mapsto\alpha_{m+1}\quad,\quad\ldots\quad,\quad\alpha_i\mapsto\alpha_{\i}\quad,
\end{align}
hence 
\begin{align}
\delta^2(\theta^\alpha)=&\sum_{i=1}^\infty\frac{1}{(i-1)!}(-1)^{|x_{\alpha_1}\cdots x_{\alpha_i}|_\rev+\sum_{l=1}^i|x_{\alpha_l}|l+|\theta^{\alpha_1}\cdots\theta^{\alpha_{\i}}|}\ell^\alpha_{\alpha_1,\ldots,\alpha_i}\theta^{\alpha_1}\cdots\theta^{\alpha_{\i}}\delta(\theta^{\alpha_i})\nn\\
=&\sum_{i=1}^\infty\frac{(-1)^{|x_{\alpha_1}\cdots x_{\alpha_{\i}}|_\rev+|x_\beta|(i+|x_{\alpha_1}\cdots x_{\alpha_{\i}}|)+\sum_{l=1}^{\i}|x_{\alpha_l}|\l+\i}}{(i-1)!}\ell^\alpha_{\alpha_1,\ldots,\alpha_{\i},\beta}\theta^{\alpha_1}\cdots\theta^{\alpha_{\i}}\delta(\theta^\beta)\quad.\label{eq:pre-reindex delta^2 theta^i}
\end{align}
We write
\begin{align}
\delta(\theta^\beta):=\sum_{j=1}^\infty\frac{1}{j!}(-1)^{|x_{\alpha_i}\cdots x_{\alpha_{\i+j}}|_\rev+\sum_{l=i}^{\i+j}|x_{\alpha_l}|(\l+i)}\ell^\beta_{\alpha_i,\ldots,\alpha_{\i+j}}\theta^{\alpha_i}\cdots\theta^{\alpha_{\i+j}}
\end{align}
so that \eqref{eq:pre-reindex delta^2 theta^i} is given by
\begin{align}\label{eq:CE proof-latter equality}
\sum_{\begin{smallmatrix}i=1\\j=1
\end{smallmatrix}}^\infty\frac{(-1)^{\i+j+|x_{\alpha_1}\cdots x_{\alpha_{\i+j}}|_\rev+j(\i+|x_{\alpha_1}\cdots x_{\alpha_{\i}}|)+\sum_{l=1}^{\i+j}|x_{\alpha_l}|\l}}{(i-1)!j!}\ell^\alpha_{\alpha_1,\ldots,\alpha_{\i},\beta}\ell^\beta_{\alpha_i,\ldots,\alpha_{\i+j}}\theta^{\alpha_1}\cdots\theta^{\alpha_{\i+j}}\quad.
\end{align}
Recalling that $\theta^\alpha:=\varsigma x^\alpha$ and pulling all the shift factors $\varsigma$ to the right, \eqref{eq:CE proof-latter equality} becomes
\begin{align}\label{eq:shift factors to the right}
\sum_{\begin{smallmatrix}i=1\\j=1
\end{smallmatrix}}^\infty\frac{(-1)^{\i+j+|x_{\alpha_1}\cdots x_{\alpha_{\i+j}}|_\rev+|x_{\alpha_1}\cdots x_{\alpha_{\i+j}}|+j(\i+|x_{\alpha_1}\cdots x_{\alpha_{\i}}|)}}{(i-1)!j!}\ell^\alpha_{\alpha_1,\ldots,\alpha_{\i},\beta}\ell^\beta_{\alpha_i,\ldots,\alpha_{\i+j}}x^{\alpha_1}\cdots x^{\alpha_{\i+j}}\varsigma^{\i+j}
\end{align}
which vanishes if and only if \eqref{eq:indexed Jacobi} is satisfied.
\end{proof}

\subsection{The symmetric monoidal dg-equivalence}\label{subsec:equivalence}
Associated with any CDGA $A$ is its dg-category ${}_A\dgMod$ of \textbf{$A$-dg-modules}.
More explicitly, an object $M=(M,\triangleright)$ in this category is a cochain complex $M=(M^\sharp,d)\in\Ch$ endowed with a cochain map
\begin{align}\label{eq:left action}
\triangleright:A\otimes M\to M\qquad,\qquad a\otimes m\mapsto a\triangleright m=:am\quad,
\end{align}
(called \textbf{action}) which satisfies the 
usual module axioms of associativity,
\begin{subequations}
\begin{align}\label{eq:module assoc}
(a\,a')m=a(a'm)\quad,
\end{align}
and unitality,
\begin{align}\label{eq:module unital}
\eta m=m\quad.    
\end{align}
\end{subequations}
The left action \eqref{eq:left action} being a cochain map means that we have 
\begin{align}\label{eq:act being cochain}
d(am)=(\delta a)m+(-1)^{|a|}a(d m)\quad.
\end{align}
The dg-module
of morphisms between two objects $M,N\in {}_A\dgMod$ is the subcomplex
\begin{subequations}\label{eqn:hom_A}
\begin{align}
{}_A\dgMod(M,N)\,\subseteq \, \Ch(M,N)
\end{align}
of the internal hom in $\Ch$ containing the \textbf{$A$-linear} maps $f\in\Ch(M,N)$, i.e.
\begin{align}\label{eq:def of A-linearity}
f(am)=(-1)^{|f||a|}af(m)
\end{align}
with the action given by
\begin{align}\label{eqn:inner hom action}
(af)(m):=af(m)\quad.
\end{align}
\end{subequations}
The dg-category ${}_A\dgMod$ has a compatible symmetric monoidal structure
with: 
\begin{enumerate}
\item[(i)] Monoidal product given by the relative tensor product,
\begin{align}\label{eqn:otimesA}
M\otimes_AN:=\colim\bigg(
\xymatrix@C=8em{
M\otimes N ~&~ \ar@<1ex>[l]^-{\id_M\,\otimes\, \triangleright}\ar@<-1ex>[l]_-{(\triangleright\,\otimes\, \id_N)\,(s_{M,A}^{}\,\otimes\,\id_N)}M\otimes A\otimes N
}
\bigg)\in\,{}_A\dgMod\,\,.
\end{align}
\item[(ii)] Monoidal unit given by the base CDGA, $A=(A,\mu)\in {}_A\dgMod$.
\item[(iii)] Symmetric braiding given by the Koszul sign rule \eqref{eq:Koszul sign},
\begin{align}\label{eqn:gammaA}
s_{M,N}:M\otimes_AN\to N\otimes_AM\qquad,\qquad m\otimes_An\mapsto(-1)^{|m||n|}n\otimes_Am\quad.
\end{align}
\end{enumerate}

\begin{defi}\label{def:semifreedgMod}
An $A$-dg-module $M=(M,\triangleright)$ is \textbf{semi-free}
if its underlying $A^\sharp$-graded-module $M^\sharp = (M^\sharp,\triangleright )$ is free, i.e. there exists a graded vector space $V$ such that the following isomorphism holds between $A^\sharp$-graded-modules,
\begin{subequations}
\begin{align}
M^\sharp\cong A^\sharp\otimes V
\end{align}
where the left action is simply given by
\begin{align}
a(a'\otimes v)=(a\,a')\otimes v=:a\,a'v\quad.
\end{align}
\end{subequations}
\end{defi}
\begin{rem}\label{rem:sf dgmod diff determined}
Analogous to Remark \ref{rem:diff determined on sfCDGA}, the fact that the action is cochain \eqref{eq:act being cochain} means that the differential $d$ on a semi-free $A$-dg-module $M^\sharp\cong A^\sharp\otimes V$ is determined by its restriction $d(v)$ to the generators. Hence, one can prescribe such differential by its values $d(v)$ on generators such that $|d(v)|=|v|+1$ and $d^2(v)=0$.
\end{rem}

\begin{rem}\label{rem:semifreedgMod}
The monoidal structure on ${}_A\dgMod$
restricts to one on the full subcategory ${}_A\dgMod^\sf\subseteq{}_A\dgMod$.
Indeed, the monoidal unit $A\in{}_A\dgMod$ is semi-free (for the simple reason that $A^\sharp\cong A^\sharp\otimes\bbK$) and,
given any two semi-free $A$-dg-modules
$M,N\in {}_A\dgMod^\sf$ with 
$M^\sharp\cong A^\sharp\otimes V$ and $N^\sharp\cong A^\sharp\otimes W$, 
we have isomorphisms between graded $A^\sharp$-modules,
\begin{align}
\big(M\otimes_A N\big)^\sharp\,\cong\,A^\sharp\otimes\big(V\otimes W\big)
\end{align}
which implies that $M\otimes_A N\in {}_A\dgMod^\sf$ is semi-free.
\end{rem}

\begin{rem}\label{rem:curved representation -> no mon prod}
Given a finite-dimensional homotopy Lie algebra $\g$, consider the Chevalley-Eilenberg algebra $\CE_\g$ of Proposition \ref{propo:CE of fd L infinity algebra}. Given a representation $\left\{\rho_V^i:\g^{\otimes\i}\otimes V\to V[2-i]\right\}_{i\geq1}$ as in \eqref{subeq:Allocca rep}, we set
\begin{equation}
\rho^i_V(x_{\alpha_1},\ldots,x_{\alpha_{\i}},v):=\rho^i_{\alpha_1,\ldots,\alpha_{\i}}(v)\quad.
\end{equation}
We would like to construct the Chevalley-Eilenberg dg-module as a semi-free $\CE_\g$-dg-module with the following underlying $\CE^\sharp_\g$-graded-module,
\begin{align}
\left(\prod_{i\in\bbN}\Sym^i(\g^*[-1])\right)\otimes V\quad.
\end{align}
Analogous to Remark \ref{rem:need completed semi-free CDGAs}, the problem is that the differential $d$ on generators $v\in V$ will be given, schematically, by 
\begin{equation}
dv\sim\sum_{i=1}^\infty\theta^{\alpha_1}\cdots\theta^{\alpha_{\i}}\rho^i_{\alpha_1,\ldots,\alpha_{\i}}(v)\quad.
\end{equation}
\end{rem}
\begin{defi}\label{def:completed semi-free dg-module}
Given a completed semi-free CDGA $A$ with generators $U$, a \textbf{completed semi-free $A$-dg module} $M$ is one for which there exists a graded vector space $V$ such that the following isomorphism holds between $A^\sharp$-graded-modules,
\begin{align}
M^\sharp\cong\prod_{i\in\bbN}\left(\Sym^i(U)\otimes V\right)\quad.
\end{align}
\end{defi}
We recall $\pi_1^k=(-1)^{\k}\ell^k$ and \eqref{eq:rewritten weak action property} hence
\begin{align}\label{eq:curved action index}
\sum_{\i+j=k}(-1)^{i\j}\rho^i_V\big[(1_{\i-1}\otimes\ell^j\otimes1_V)(\Sigma_{\i-1,j}\otimes1_V)+(1_{\j}\otimes\rho^j_V)(\Sigma_{\i,\j}\otimes1_V)\big]=0\quad.
\end{align}
\begin{propo}\label{propo:CE of rep}
Given a finite-dimensional homotopy Lie algebra $\g=(\g,\ell)$ and a representation $V=(V,\rho_V)$, the differential $d$ of the Chevalley-Eilenberg dg-module is given on generators $v\in V$ of $(\CE_\g^V)^\sharp$ by
\begin{align}\label{eq:d_V tens}
dv:=\sum_{i=1}^\infty\frac{1}{\i!}(-1)^{|x_{\alpha_1}\cdots x_{\alpha_{\i}}|_\rev+\sum_{l=1}^{\i}|x_{\alpha_l}|\l}\theta^{\alpha_1}\cdots\theta^{\alpha_{\i}}\rho^i_{\alpha_1,\ldots,\alpha_{\i}}(v)\quad.
\end{align}
\end{propo}
\begin{proof}
We express the square of the differential $d$ on a generator $v\in V$,
\begin{align}
d^2(v)=\,&\sum_{i=2}^\infty\frac{1}{\i!}(-1)^{|x_{\alpha_1}\cdots x_{\alpha_{\i-1}}|_\rev+|x_\alpha|(i+|x_{\alpha_1}\cdots x_{\alpha_{\i-1}}|)+\sum_{l=1}^{\i-1}|x_{\alpha_l}|\l}\delta(\theta^{\alpha_1}\cdots\theta^{\alpha_{\i-1}}\theta^\alpha)\rho^i_{\alpha_1,\ldots,\alpha_{\i-1},\alpha}(v)\nn\\&+\sum_{j=1}^\infty\frac{1}{\j!}(-1)^{|x_{\alpha_1}\cdots x_{\alpha_{\j}}|_\rev+\sum_{l=1}^{\j}|x_{\alpha_l}|l+\j}\theta^{\alpha_1}\cdots\theta^{\alpha_{\j}}d\big(\rho^j_{\alpha_1,\ldots,\alpha_{\j}}(v)\big)\quad.\label{eq:dd^2(v) first expression}
\end{align}
The first term of \eqref{eq:dd^2(v) first expression} can be handled by using the steps \eqref{eq:delta^2 theta^i}-\eqref{eq:CE proof-latter equality} thus \eqref{eq:dd^2(v) first expression} is given by
\begin{align}
&\sum_{\begin{smallmatrix}i=2\\j=1
\end{smallmatrix}}^\infty\frac{1}{(\i-1)!j!}(-1)^{|x_{\alpha_1}\cdots x_{\alpha_{\i+\j}}|_\rev+j|x_{\alpha_1}\cdots x_{\alpha_{\i-1}}|+i\j+\sum_{l=1}^{\i+\j}|x_{\alpha_l}|l}\theta^{\alpha_1}\cdots\theta^{\alpha_{\i+\j}}\rho^i_{\alpha_1,\ldots,\alpha_{\i-1},\beta}(v)\ell^\beta_{\alpha_{\i},\ldots,\alpha_{\i+\j}}\nn\\&+\sum_{\begin{smallmatrix}i=1\\j=1
\end{smallmatrix}}^\infty\frac{(-1)^{|x_{\alpha_1}\cdot\cdot x_{\alpha_{\i+\j}}|_\rev+i(\j+|x_{\alpha_1}\cdot\cdot x_{\alpha_{\j}}|)+\sum_{l=1}^{\i+\j}|x_{\alpha_l}|\l}}{\i!\j!}\theta^{\alpha_j}\cdot\cdot\theta^{\alpha_{\i+\j}}\theta^{\alpha_1}\cdot\cdot\theta^{\alpha_{\j}}\rho^i_{\alpha_j..\alpha_{\i+\j}}\big[\rho^j_{\alpha_1..\alpha_{\j}}(v)\big]\label{eq:dd^2(v) second expression}
\end{align}
but we simply reindex the second sum then pull all the shift factors to the right as in \eqref{eq:shift factors to the right}, 
\begin{align}
&\sum_{\begin{smallmatrix}i=2\\j=1
\end{smallmatrix}}^\infty\frac{(-1)^{|x_{\alpha_1}\cdots x_{\alpha_{\i+\j}}|_\rev+j|x_{\alpha_1}\cdots x_{\alpha_{\i-1}}|+i\j}}{(\i-1)!j!}x^{\alpha_1}\cdots x^{\alpha_{\i+\j}}\varsigma^{\i+\j}\rho^i_{\alpha_1,\ldots,\alpha_{\i-1},\beta}(v)\ell^\beta_{\alpha_{\i},\ldots,\alpha_{\i+\j}}\nn\\&+\sum_{i,j=1}^\infty\frac{(-1)^{|x_{\alpha_1}\cdot\cdot x_{\alpha_{\i+\j}}|_\rev+j|x_{\alpha_1}\cdot\cdot x_{\alpha_{\i}}|+i\j}}{\i!\j!}x^{\alpha_1}\cdots x^{\alpha_{\i+\j}}\varsigma^{\i+\j}\rho^i_{\alpha_1,\ldots,\alpha_{\i}}\big(\rho^j_{\alpha_i,\ldots,\alpha_{\i+\j}}(v)\big)\,.\label{eq:dd^2(v) last expression}
\end{align}
\eqref{eq:dd^2(v) last expression} vanishes if and only if \eqref{eq:curved action index} is satisfied.
\end{proof}
\begin{propo}\label{propo:constructed sym mon dg-equivalence}
Given $f=\left\{f^i:\g^{\otimes\i}\otimes U\to V\big[|f|-\i\big]\right\}_{i\geq1}:U\s V\big[|f|\big]$, we define the $\CE_\g$-linear map
\begin{subequations}
\begin{equation}
\CE_\g^f:\CE_\g^U\to\CE_\g^V\big[|f|\big]
\end{equation}
on module generators $u\in U$ as 
\begin{equation}\label{eq:CE(frak g,f)(u):=...}
\CE_\g^f(u):=\sum_{i=1}^\infty\frac{1}{\i!}(-1)^{|x_{\alpha_1}\cdots x_{\alpha_{\i}}|_\rev+\sum_{l=1}^{\i}|x_{\alpha_l}|(l+|f|)}\theta^{\alpha_1}\cdots\theta^{\alpha_{\i}}f^i_{\alpha_1,\ldots,\alpha_{\i}}(u)
\end{equation}
This defines a symmetric strict monoidal dg-equivalence $\CE_\g:\Rep\simeq{}_{\CE_\g}\dgMod^\sf_\Pi$.
\end{subequations}
\end{propo}
\begin{proof}
Let us first prove functoriality; we obviously have $\CE_\g^{\oone_U}=\id_{\CE_\g^U}$ and, if we are also given $g:V\s W\big[|g|\big]$, we have 
\begin{align}
\CE_\g^{gf}(u)\overset{\eqref{subeq:juxtaposition of intertwiners}}{=}&\sum_{i=1}^\infty\frac{1}{\i!}(-1)^{|x_{\alpha_1}\cdots x_{\alpha_{\i}}|_\rev+\sum_{l=1}^{\i}|x_{\alpha_l}|(l+|g|+|f|)}\theta^{\alpha_1}\cdots\theta^{\alpha_{\i}}\sum_{\j+k=i}\sum_{\sigma\in\S_{\j,\k}}\varepsilon_\sigma\label{eq:CE of bullet prod}\\
&\qquad(-1)^{(|g|-\j)\k+(|f|-\k)|x_{\alpha_{\sigma(1)}}\cdots x_{\alpha_{\sigma(\j)}}|}g^j_{\alpha_{\sigma(1)},\ldots,\alpha_{\sigma(\j)}}\left(f^k_{\alpha_{\sigma(j)},\ldots,\alpha_{\sigma(\i)}}(u)\right)\nn
\end{align}
and we compare this with 
\begin{align*}
\CE_\g^g\CE_\g^f(u)\overset{\eqref{eq:def of A-linearity}}{=}&\sum_{j,k=1}^\infty\frac{1}{\j!\k!}(-1)^{|g|(|x_{\alpha_1}\cdots x_{\alpha_{\k}}|+\k)+|x_{\alpha_1}\cdots x_{\alpha_{\k}}|_\rev+\sum_{l=k}^{\j+\k}|x_{\alpha_l}|(l+\k+|g|)+|x_{\alpha_k}\cdots x_{\alpha_{\j+\k}}|_\rev}\\&\qquad(-1)^{\sum_{l=1}^{\k}|x_{\alpha_l}|(l+|f|)}\theta^{\alpha_1}\cdots\theta^{\alpha_{\j+\k}}g^j_{\alpha_k,\ldots,\alpha_{\j+\k}}\left(f^k_{\alpha_1,\ldots,\alpha_{\k}}(u)\right)\\
\overset{\eqref{eq:CDGA commutativity}}{=}\,\,
&\sum_{j,k=1}^\infty\frac{1}{\j!\k!}(-1)^{|x_{\alpha_1}\cdots x_{\alpha_{\j+\k}}|_\rev+\sum_{l=1}^{\j+\k}|x_{\alpha_l}|(l+|g|+|f|)}\theta^{\alpha_1}\cdots\theta^{\alpha_{\j+\k}}\\
&\qquad(-1)^{(|g|-\j)\k+(|f|-\k)|x_{\alpha_1}\cdots x_{\alpha_{\j}}|}g^j_{\alpha_1,\ldots,\alpha_{\j}}\left(f^k_{\alpha_j,\ldots,\alpha_{\i}}(u)\right)\numberthis
\end{align*}
which recovers \eqref{eq:CE of bullet prod}. Now we show \textit{dg}-functoriality; we consider $\CE_\g^{\pb{\rho}{f}}(u)$ via \eqref{eq:arity component of [rho,f]}, 
\begin{align*}
\sum_{i=1}^\infty&\frac{1}{\i!}(-1)^{|x_{\alpha_1}\cdots x_{\alpha_{\i}}|_\rev+\sum_{l=1}^{\i}|x_{\alpha_l}|\left(\l+|f|\right)}\theta^{\alpha_1}\cdots\theta^{\alpha_{\i}}\sum_{\j+k=i}\Bigg(\sum_{\sigma\in\S_{\j,\k}}\varepsilon_\sigma(-1)^{j\k+(|f|-\k)|x_{\alpha_{\sigma(1)}}\cdots x_{\alpha_{\sigma(\j)}}|}\\
&\rho^j_{\alpha_{\sigma(1)},\ldots,\alpha_{\sigma(\j)}}\left(f^k_{\alpha_{\sigma(j)},\ldots,\alpha_{\sigma(\i)}}(u)\right)+\sum_{\sigma\in\S_{j,\k-1}}\varepsilon_\sigma(-1)^{\widetilde{|f|}j}\ell^\alpha_{\alpha_{\sigma(1)},\ldots,\alpha_{\sigma(j)}}f^k_{\alpha,\alpha_{\sigma(j+1)},\ldots,\alpha_{\sigma(\i)}}(u)
\\&-\sum_{\sigma\in\S_{\k,\j}}\varepsilon_\sigma(-1)^{\j\k+j(|f|+|x_{\alpha_{\sigma(1)}}\cdots x_{\alpha_{\sigma(\k)}}|)}f^k_{\alpha_{\sigma(1)},\ldots,\alpha_{\sigma(\k)}}\Big(\rho^j_{\alpha_{\sigma(k)},\ldots,\alpha_{\sigma(\i)}}(u)\Big)\Bigg)\numberthis\label{eq:CE is dg-functor LHS}
\end{align*}
and we compare this to $d\left(\CE_\g^f(u)\right)-(-1)^{|f|}\CE_\g^f(du)$,
\begin{align*}
&\sum_{k=2}^\infty\frac{1}{(\k-1)!}(-1)^{|x_\alpha x_{\alpha_1}\cdots x_{\alpha_{\k-1}}|_\rev+|x_\alpha|\widetilde{|f|}+\sum_{l=1}^{\k-1}|x_{\alpha_l}|(\l+|f|)}\delta(\theta^\alpha)\theta^{\alpha_1}\cdots\theta^{\alpha_{\k-1}}f^k_{\alpha,\alpha_1,\ldots,\alpha_{\k-1}}(u)\\
&+\sum_{k=1}^\infty\frac{1}{\k!}(-1)^{|x_{\alpha_1}\cdots x_{\alpha_{\k}}|_\rev+\sum_{l=1}^{\k}|x_{\alpha_l}|(l+|f|)+|\theta^{\alpha_1}\cdots\theta^{\alpha_{\k}}|}\theta^{\alpha_1}\cdots\theta^{\alpha_{\k}}d\left(f^k_{\alpha_1,\ldots,\alpha_{\k}}(u)\right)\\
&-(-1)^{|f|}\CE_\g^f\left(\sum_{j=1}^\infty\frac{1}{\j!}(-1)^{|x_{\alpha_1}\cdots x_{\alpha_{\j}}|_\rev+\sum_{l=1}^{\j}|x_{\alpha_l}|\l}\theta^{\alpha_1}\cdots\theta^{\alpha_{\j}}\rho^j_{\alpha_1,\ldots,\alpha_{\j}}(u)\right)\numberthis
\end{align*}
which becomes, upon using \eqref{eq:CE differential on generator}, \eqref{eq:d_V tens} and \eqref{eq:CE(frak g,f)(u):=...},
\begin{align*}
&\sum_{\begin{smallmatrix}k=2\\j=1
\end{smallmatrix}}^\infty\frac{1}{\j!(\k-1)!}(-1)^{(|x_{\alpha_1}\cdots x_{\alpha_{\k-1}}|+\widetilde{|f|})(|x_{\alpha_{\k}}\cdots x_{\alpha_{\j+\k}}|+j)+|x_{\alpha_1}\cdots x_{\alpha_{\k-1}}|_\rev+\sum_{l=1}^{\k-1}|x_{\alpha_l}|(\l+|f|)}\\&\qquad\qquad\qquad(-1)^{|x_{\alpha_{\k}}\cdots x_{\alpha_{\j+\k}}|_\rev+\sum_{l=\k}^{\j+\k}|x_{\alpha_l}|(l+k)}\ell^\alpha_{\alpha_{\k},\ldots,\alpha_{\j+\k}}\theta^{\alpha_{\k}}\cdots\theta^{\alpha_{\j+\k}}\theta^{\alpha_1}\cdots\theta^{\alpha_{\k-1}}f^k_{\alpha,\alpha_1,\ldots,\alpha_{\k-1}}(u)\\
&+\sum_{\begin{smallmatrix}k=1\\j=1
\end{smallmatrix}}^\infty\frac{1}{\j!\k!}(-1)^{|x_{\alpha_1}\cdots x_{\alpha_{\k}}|_\rev+\sum_{l=1}^{\k}|x_{\alpha_l}|(l+|f|)+|x_{\alpha_1}\cdots x_{\alpha_{\k}}|+\k+|x_{\alpha_k}\cdots x_{\alpha_{\j+\k}}|_\rev+\sum_{l=k}^{\j+\k}|x_{\alpha_l}|(l+k)}\\
&\qquad\qquad\qquad\quad\theta^{\alpha_1}\cdots\theta^{\alpha_{\k}}\theta^{\alpha_k}\cdots\theta^{\alpha_{\j+\k}}\rho^j_{\alpha_k,\ldots,\alpha_{\j+\k}}\left(f^k_{\alpha_1,\ldots,\alpha_{\k}}(u)\right)\\
&-\sum_{\begin{smallmatrix}j=1\\k=1
\end{smallmatrix}}^\infty\frac{1}{\j!\k!}(-1)^{|x_{\alpha_1}\cdots x_{\alpha_{\j}}|_\rev+\sum_{l=1}^{\j}|x_{\alpha_l}|\l+|f|(1+|\theta^{\alpha_1}\cdots\theta^{\alpha_{\j}}|)+|x_{\alpha_j}\cdots x_{\alpha_{\j+\k}}|_\rev+\sum_{l=j}^{\j+\k}|x_{\alpha_l}|(\l+j+|f|)}\\
&\qquad\qquad\qquad\quad\theta^{\alpha_1}\cdots\theta^{\alpha_{\j}}\theta^{\alpha_j}\cdots\theta^{\alpha_{\j+\k}}f^k_{\alpha_j,\ldots,\alpha_{\j+\k}}\left(\rho^j_{\alpha_1,\ldots,\alpha_{\j}}(u)\right)\numberthis\label{eq:CE is dg-functor RHS}
\end{align*}
but we recover \eqref{eq:CE is dg-functor LHS} by reindexing \eqref{eq:CE is dg-functor RHS} and shuffling terms as always. Essential surjectivity is proven by Proposition \ref{propo:CE of rep}, faithfulness is obvious and fullness follows from the above. Let us now prove that the strict monoidal structure is strictly preserved, on objects we have
\begin{align*}
d(u\otimes_{\CE_\g}v)=\,&(du)\otimes_{\CE_\g}v+(-1)^{|u|}u\otimes_{\CE_\g}dv\\=\,&\sum_{i=1}^\infty\frac{1}{\i!}(-1)^{|x_{\alpha_1}\cdots x_{\alpha_{\i}}|_\rev+\sum_{l=1}^{\i}|x_{\alpha_l}|\l}\theta^{\alpha_1}\cdots\theta^{\alpha_{\i}}\\&\qquad\Big[\rho^i_{\alpha_1,\ldots,\alpha_{\i}}(u)\otimes_{\CE_\g}v+(-1)^{|u|(1+|\theta^{\alpha_1}\cdots\theta^{\alpha_{\i}}|)}u\otimes_{\CE_\g}\rho^i_{\alpha_1,\ldots,\alpha_{\i}}(v)\Big]\\=\,&\sum_{i=1}^\infty\frac{1}{\i!}(-1)^{|x_{\alpha_1}\cdots x_{\alpha_{\i}}|_\rev+\sum_{l=1}^{\i}|x_{\alpha_l}|\l}\theta^{\alpha_1}\cdots\theta^{\alpha_{\i}}\rho^i_{\alpha_1,\ldots,\alpha_{\i}}(u\otimes v)\\=\,&d(u\otimes v)\numberthis
\end{align*}
thus $\CE_\g^U\otimes_{\CE_\g}\CE_\g^V=\CE_\g^{U\odot V}$. On morphisms $f:U\s U'\big[|f|\big]$ and $g:V\s V'\big[|g|\big]$, we have
\begin{align*}
\CE_\g^{f\odot g}(u\otimes v)=&\sum_{i=1}^\infty\frac{1}{\i!}(-1)^{|x_{\alpha_1}\cdots x_{\alpha_{\i}}|_\rev+\sum_{l=1}^{\i}|x_{\alpha_l}|(l+|f|+|g|)}\theta^{\alpha_1}\cdots\theta^{\alpha_{\i}}\sum_{\j+k=i}\sum_{\sigma\in\S_{\j,\k}}\varepsilon_\sigma(-1)^{(|f|-\j)\k}\\&(-1)^{|u||x_{\alpha_{\sigma(j)}}\cdot\cdot x_{\alpha_{\sigma(\i)}}|+(|g|-\k)|x_{\alpha_{\sigma(1)}}\cdot\cdot x_{\alpha_{\sigma(\j)}}u|}f^j_{\alpha_{\sigma(1)},..,\alpha_{\sigma(\j)}}(u)\otimes g^k_{\alpha_{\sigma(j)},..,\alpha_{\sigma(\i)}}(v)\numberthis\label{eq:CE(frak g,f odot g):=...}
\end{align*}
and we compare this with the following expression for $\CE_\g^f(u)\otimes_{\CE_\g}(-1)^{|g||u|}\CE_\g^g(v)$,
\begin{align*}
\sum_{\begin{smallmatrix}j=1\\k=1
\end{smallmatrix}}^\infty\frac{1}{\j!\k!}(-1)&^{|x_{\alpha_1}\cdots x_{\alpha_{\j}}|_\rev+\sum_{l=1}^{\j}|x_{\alpha_l}|(l+|f|)+|x_{\alpha_j}\cdots x_{\alpha_{\j+\k}}|_\rev+\sum_{l=j}^{\j+\k}|x_{\alpha_l}|(l+\j+|g|)+|g||u|}\\&\theta^{\alpha_1}\cdots\theta^{\alpha_{\j}}f^j_{\alpha_1,\ldots,\alpha_{\j}}(u)\otimes_{\CE_\g}\theta^{\alpha_j}\cdots\theta^{\alpha_{\j+\k}}g^k_{\alpha_j,\ldots,\alpha_{\j+\k}}(v)\numberthis
\end{align*}
which recovers \eqref{eq:CE(frak g,f odot g):=...} upon shuffling terms. Finally, we obviously have preservation of monoidal units $\CE_\g^\bbK=\CE_\g$ and preservation of the symmetric braiding $\CE_\g^{\gamma_{U,V}}=s_{\CE_\g^U,\CE_\g^V}$.
\end{proof}

\end{document}